\documentclass[reqno]{amsart}
\usepackage{amssymb,amsmath}
\usepackage{amsthm}
\usepackage{color,graphicx}
\usepackage{xcolor}
\usepackage{hyperref}
\usepackage{verbatim}

\usepackage{stmaryrd}

\newcommand{\lanxN}{\langle x \rangle_N}
\newcommand{\ha}{\hat{\sigma}}
\newcommand{\hu}{\hat{\phi}}
\newcommand{\va}{\varphi}
\newcommand{\li}{\|\p_\xi^2 \ha\|_{L^\infty_\xi}}
\newcommand{\lii}{\|\p_\xi^3 \ha\|_{L^\infty_\xi}}
\newcommand{\les}{\lesssim}
\newcommand{\s}{\sgn(\xi)}
\newcommand{\si}{\sgn(\xi)}

\newcommand{\lan}{\langle \xi \rangle}

\newcommand{\lanx}{\langle x \rangle}

\newcommand{\D}{D^{\theta}_\xi}

\newcommand{\Dta}{\mathcal{D}^{\theta}_\xi}

\newcommand{\h}{\mathcal H}
\newcommand{\z}{\mathcal Z}

\newcommand{\Dt}{\mathcal{D}^{\theta}_\xi}

\newcommand{\R}{\mathbb R}

\newcommand{\Z}{\mathbb Z}

\newcommand{\p}{\partial}

\newcommand{\sgn}{\text{sgn}}

\numberwithin{equation}{section}

\newtheorem{theorem}{Theorem}[section]
\newtheorem{proposition}[theorem]{Proposition}
\newtheorem{remark}[theorem]{Remark}
\newtheorem{lemma}[theorem]{Lemma}

\begin{document}

\title[The (DGBO) equation ]{On decay of the solutions for the dispersion generalized-Benjamin-Ono and Benjamin-Ono equations}



\author[A. Cunha]{Alysson Cunha}
\address{Instituto de Matem\'atica e Estat\'istica(IME).
Universidade Federal de Goi\'as(UFG), Campus Samambaia, 131, 74001-970, Goi\^ania, Bra\-zil}
\email{alysson@ufg.br}


\subjclass[2020]{35A01, 35B60, 35Q53, 35R11}

\keywords{DGBO equation, BO equation, Initial-value problem, Well-posedness, Weighted spaces}

\begin{abstract}

We show that uniqueness results of the kind those obtained for KdV and Schr\"odinger equations (\cite{EKPV}, \cite{NahasPonce}), are not valid for the dispersion generalized-Benjamin-Ono equation in the weighted Sobolev spaces

$$H^s(\R)\cap L^2(x^{2r}dx),$$

for appropriated $s$ and $r$. 

In particular, we obtain that the uniqueness result proved for the dispersion generalized-Benjamin-Ono equation (\cite{FLP1}), is not true for all pairs of solutions $u_1\neq 0$ and $u_2\neq 0$. To achieve these results we employ the techniques present in our recent work \cite{ap1}. We also improve some Theorems established for the dispersion generalized-Benjamin-Ono equation and for the Benjamin-Ono equation (\cite{FLP1}, \cite{GermanPonce}).

\end{abstract}
 
\maketitle

\section{Introduction}\label{introduction}

This paper is concerned with the initial-value problem (IVP) associated with the dispersion
generalized-Benjamin-Ono (DGBO) equation,
\begin{equation}\label{gbo}
\begin{cases}
u_{t}+D^{a+1}_x \partial_{x}u+uu_{x}=0, \;\;x,t\in\R, \quad a \in (0,1), \\
u(x,0)=\phi(x),
\end{cases}
\end{equation}
where $D^{a+1}_x$ stands for the fractional derivative of order $a+1$ with respect to the variable $x$ and is defined, via your Fourier transform, as $D^{a+1}_x f(x)=(|\xi|^{a+1}\widehat{f})^\vee(x)$.

For $a=1$, the IVP \eqref{gbo} becomes the well-known IVP for the KdV equation, for which the local and global well-posedness (LWP and GWP, respectively) in Sobolev spaces already widely studied, see \cite{KPV1,KPV2,Bona}, and references therein. About persistence in weighted Sobolev spaces, in \cite{Kato} it was proved that the Schwartz space is preserved by the flow of the KdV equation.

When $a=0$ we obtain the IVP for the Benjamin-Ono equation (BO)

\begin{equation}\label{bo}
\begin{cases}
u_{t}+\h\partial_{x}^2 u+uu_{x}=0, \;\;x,t\in\R, \\
u(x,0)=\phi(x),
\end{cases}
\end{equation}
where $\h$ stands for the Hilbert transform.

With respect to the well-posedness in the Sobolev spaces, the IVP \eqref{bo} has been largely studied recently, see for example \cite{Ab,Koch,KenigKoenig,Tao,Mpilot,Burq,GPonce,Io}. Through the work of Iorio \cite{Iorio}, the flow of the BO equation, in general polynomial decay is not preserved in weighted Sobolev spaces. The extension for fractional weights can be found in the paper of Ponce and Fonseca \cite{FonsecaPonce}.

Lately, the idea of improving theorems in weighted Sobolev spaces for the Benjamin-Ono equation, by assuming a higher spacial decay on the initial data, can be found in \cite{flores}.  

Recently Kenig, Ponce and Vega (\cite{kpv}) obtained uniqueness results of solutions of the IVP \eqref{bo}. Roughly, they proved that for solutions $u_1$ and $u_2$ in a suitable class of functions, defined in $\R \times [0,T]$, which agree in an open set $\Omega \subset \R \times [0,T]$, we have $u_1\equiv u_2$. 

We recall that the solutions of the IVP \eqref{bo} satisfy the following conservation laws (see \cite{FonsecaPonce})

\begin{equation}\label{I1}
I_1(u)=\int_{\R} u(x,t)dx, 
\end{equation}

\begin{equation}\label{I2}
I_2(u)=\int_{\R} u^2(x,t)dx 
\end{equation}

and

\begin{equation}\label{I3}
I_3(u)=\int_{\R} \Big(|D^{1/2}u|^2-\frac{u^3}{3}\Big)(x,t)dx,
\end{equation}

where $D=\h \p_x$.

The IVP \eqref{gbo} has been extensively studied in the last years. In particular, for the local well-posedness in Sobolev spaces $H^s(\R)$, see \cite{Guogbo, MolinetRibaud, Sebastian} and references therein. In \cite{MolinetSau} the authors proved that, as in the case of BO equation, the IVP associated to the DGBO equation the application data-solution from $H^s (\R)$ to $C([0,T];H^s(\R))$ fails to be locally $C^2$ at the origin for all $T>0$ and all $s\in \R$. Hence, the IVP \eqref{gbo} cannot be solved by an argument based only on the contraction principle.

For the DGBO equation, lately Kenig, Pilot, Ponce and Vega (see \cite{KPPV}) proved uniqueness results, similarly to obtained by the authors in \cite{kpv}.

Real solutions sufficiently smooth of the IVP \eqref{gbo} satisfies (see \cite{FLP1})

\begin{equation}\label{norm}
\frac{d}{dt}\int u^2(x,t)dx=0.
\end{equation}

Let us start by recalling the following well-posedness results about the IVP \eqref{gbo}. See the notation of function spaces in the Section \ref{notation}.

\noindent {\bf Theorem A.}  \textit{Let $a\in(0,1)$.
\begin{itemize}
	\item[(a)] If $\phi\in Z_{s,r}$ then the solution of \eqref{gbo} satisfies $u\in C([-T,T]:Z_{s,r})$ if either
	\begin{itemize}
		\item[(i)] $s\geq 1+a$ and $r\in(0,1]$, or
		\item[(ii)] $s\geq2(1+a)$ and $r\in(1,2]$, or
		\item[(iii)] $s\geq \llbracket (r+1)^- \rrbracket(1+a)$ and $2<r<5/2+a$, with $\llbracket \cdot \rrbracket$ denoting the integer part function.
	\end{itemize}
\item[(b)] If $\phi\in \dot{Z}_{s,r}$ then the solution of \eqref{gbo} satisfies $u\in C([-T,T]:\dot{Z}_{s,r})$ provided
\begin{itemize}
	\item[(iv)] $s\geq \llbracket (r+1)^- \rrbracket(1+a)$ and $5/2+a<r<7/2+a.$\\
\end{itemize}
\end{itemize} }

The result in Theorem A is sharp in the sense that one cannot have solutions with additional decay in the spatial variable. More precisely,\\

\noindent {\bf Theorem B.}  \textit{Let $a\in(0,1)$ and assume $u\in C([-T,T]:Z_{s,(7/2+a)^-})$, with
	$$
	T>0 \qquad \mbox{and}\qquad s\geq (1+a)(7/2+a)+\frac{1-a}{2},
	$$
	is a solution of \eqref{gbo}. If there exist three different times $t_1,t_2,t_3\in[-T,T]$ such that
	$$
	u(\cdot,t_j)\in\dot{Z}_{s,7/2+a}, \quad j=1,2,3,
	$$ 
	then $u\equiv0$.}\\

Next theorem shows that the condition in Theorem B at two times is not sufficient to guarantee that $u\equiv0$.

\noindent {\bf Theorem C.}  \textit{Let $a\in(0,1)$ and assume $u\in C([-T,T]:Z_{s,(7/2+a)^-})$, with
	$$
	T>0 \qquad \mbox{and}\qquad s\geq (1+a)(7/2+\llbracket 1+2a \rrbracket/2)+\frac{1-a}{2},
	$$
	is a solution of \eqref{gbo} such that
	$$
\phi\in \dot{Z}_{s,7/2+\tilde{a}}, \quad \tilde{a}=\llbracket 1+2a \rrbracket/2, \quad\mbox{and}\quad \int_{\R}x\phi(x)dx\neq0.
	$$ 
Then there exists $t^*\neq0$ with
$$t^*=\frac{-4}{\|\phi\|^2}\int_{-\infty}^{\infty}x\phi(x)dx$$
such that $u(t^*)\in \dot{Z}_{7/2+\tilde{a}}$. }\\


The following theorem is about the IVP \eqref{bo}. Such result shows that the uniqueness result obtained in \cite[Theorem 3]{FonsecaPonce} cannot extend for all pairs $u_1\neq 0$ and $u_2\neq 0$ of solutions of the BO equation.

\noindent {\bf Theorem D.}  \textit{There exist $u_1,u_2\in C(\R;Z_{4,2})$ solutions of the IVP \eqref{bo} such that
$$u_1 \neq u_2$$
and for any $T>0$
$$u_1-u_2 \in L^{\infty}([-T,T];Z_{4,4}).$$
 }\\

The proof of Theorems A--C can be found in the work of Fonseca, Linares, and Ponce \cite{FLP1}. The proof of Theorem D is due to the same authors and can be seen in \cite{GermanPonce}.

The result in Theorem B may be seen as a uniqueness result. Here, we are mainly concerned with the question of when this uniqueness result can be extended to any pair  of solutions $u_1$ and $u_2$ of \eqref{gbo} with $u_1\neq0$ and $u_2\neq0$. More precisely, we are interested in the following question: assume 
\begin{equation}\label{uniq3}
u_1(\cdot,t_1)-u_2(\cdot,t_1),\ u_1(\cdot,t_2)-u_2(\cdot,t_2), \ u_1(\cdot,t_3)-u_2(\cdot,t_3)\in X,
\end{equation}
for three different times $t_1,t_2,t_3$ and some function space $X$; is it true that $u_1\equiv u_2$? In our case we will take $X=Z_{s,r}$ for suitable $s$ and $r$. For the KdV equation, this question (for two different times) was addressed in the paper of Escauriaza, Kenig, Ponce and Vega \cite{EKPV}. Indeed, the authors showed if $u_1$ and $u_2$ are two solutions of the KdV equation in $C([0,1];Z_{3,1})$ with
$$
u_1(\cdot,0)-u_2(\cdot,0),\ u_1(\cdot,1)-u_2(\cdot,1)\in H^1(e^{ax_+^{3/2}}dx)
$$
for some $a>0$ then $u_1\equiv u_2$. Here $x_+=\max\{0,x\}$. For similar questions for the Schr\"odinger equation, see the work \cite{EKPV1}, of the same authors above. For the BO equation, we recall that from the work of Fonseca and Ponce \cite[Theorem 3]{FonsecaPonce} a similar uniqueness result as in Theorem B with $a=0$ also holds. On the other hand, the Theorem D above shows that the result established in \cite[Theorem 3]{FonsecaPonce} can not extend for two non-zero solutions. Hence, the question \eqref{uniq3}  for the BO equation has a negative answer.

Now, consider the question \eqref{uniq3}, for two different times. We recall that a similar result as in Theorem C with $a=0$ can be found in \cite[Theorem 2]{GermanPonce}. Hence, by this result and the Theorem C, follows that the Theorems B and \cite[Theorem 3]{FonsecaPonce} are not valid for two different times. This shows us that the question \eqref{uniq3}, for two different times, for the DGBO and BO equations has a negative answer.


Next we present our results. We start with results about the dispersion generalized Benjamin-Ono. The first one tells us that similar uniqueness properties of solutions as in the KdV and Schr\"odinger equations (see \cite{EKPV} and \cite{EKPV1}, resp.) are not true for the IVP \eqref{gbo}. By the other hand, this also shows that the uniqueness properties obtained in Theorem B can not be extended to all pair of non-zero solutions. To achieve such results we use some ideas developed in \cite{ap1}(see also \cite{dBO}).

\begin{theorem}\label{n}
Assume $a\in(0,1)$ and let $u_1$ and $u_2$ be solutions of \eqref{gbo} satisfying 
\begin{equation}\label{inicial}
u_1(0)=\phi, u_2(0)=\varphi, \phi\not =\varphi,
\end{equation} 
\begin{equation}\label{norma}
\|\phi\|=\|\va\|,
\end{equation}
\begin{equation}\label{v}
\int \phi(x)dx=\int \va(x)dx, 
\end{equation}
\begin{equation}\label{v1}
\int x\phi(x)dx=\int x\va(x)dx, 
\end{equation}
and
\begin{equation}
\phi, \va \in Z_{s,4},
\end{equation}
where $s\geq 4(1+a)+1$.

 Then $u_1\not = u_2$ and for some $T>0$
 \begin{equation}\label{u-v4}
  u_1-u_2 \in C([-T,T]; Z_{s,4}).
 \end{equation}
\end{theorem}

The idea to prove Theorem \ref{n} is the following. Let $u_1$ and $u_2$ be solutions of the IVP \eqref{gbo} with initial data $\phi$ and $\varphi$, respectively. Then $w:=u_1 - u_2$ satisfies the following variable coefficients linear IVP 
\begin{equation}\label{wgbo}
	\begin{cases}
		w_{t}+D^{a+1}_x \partial_{x}w+u_1\p_x w+w\p_x u_{2}=0, \;\;x\in\R, \;t>0, \quad a \in (0,1), \\
		w(x,0)=\phi(x)-\varphi(x).
	\end{cases}
\end{equation}

By setting $\sigma=\phi-\varphi$ and $z=\frac12 \p_x (u_1^2-u_2^2)=u_1\p_x w+w\p_x u_{2}$, the integral equation associated with  \eqref{wgbo} is given by
\begin{equation}
	\begin{split}\label{int}
		w(t)=\ &U(t)\sigma-\int_0^t U(t-\tau) z(\tau)d\tau.
	\end{split}
\end{equation}
Thus, using the Fourier transform, we manage to show that the right-hand side of \eqref{int} belongs to $C([-T,T]; Z_{s,4}).$

In the following, we improve the last theorem, by assuming additional decay on the initial data.
\begin{theorem}\label{no}
Let $\theta \in [0,1)$ and $(r,a)\in \{4\}\times (0,1/2]\cup \{5\}\times (1/2,1)$ be the such that $r+\theta<9/2+a$. Assume $u_1,u_2$ solutions of the IVP \eqref{gbo} with initial data $\phi$ and $\varphi$, respectively. Suppose that the hypothesis \eqref{inicial}--\eqref{v1} are valid and

\begin{equation}
\phi, \va \in Z_{s,r+\theta}, \;\; \mbox{where}\;\; s\geq (r+\theta)(1+a)+1.
\end{equation}

Then $u_1\not = u_2$ and for some $T>0$
 \begin{equation}\label{u-v5t}
  u_1-u_2 \in C([-T,T]; Z_{s,r+\theta}).
 \end{equation}
 
\end{theorem}

Roughly speaking, Theorem \ref{no} reduces to Theorem \ref{n} when we set $\theta=0$. The idea to prove Theorem \ref{no} is exactly the same of Theorem \ref{n} but now we need to handle with fractional derivatives.


The next theorem corresponds to an improvement of Theorem C, in which we assuming an additional decay on the initial data. Note that this result also allows us to conclude that Theorem B cannot be extended for two different times, even assuming a higher decay in the initial data.

\begin{theorem}\label{2t}
Let $\theta \in [0,1)$ and $(r,a)\in \{4\}\times (0,1/2]\cup \{5\}\times(1/2,1)$ be the such that $r+\theta<9/2+a$. Suppose that $u\in C([-T,T];\z_{s,(\frac72+a)^-})$ is a solution of the IVP \eqref{gbo}, for $s\geq (r+\theta)(1+a)+1$. If $u(0)=\phi\in \dot{Z}_{s,r+\theta}$ and 
\begin{equation*}
\int_{-\infty}^{\infty}x\phi(x)dx\neq 0,
\end{equation*}
then $$u(t^*)\in Z_{s,r+\theta},$$ 
where
$$t^*=\frac{-4}{\|\phi\|^2}\int_{-\infty}^{\infty}x\phi(x)dx.$$
\end{theorem}
The goal of the next result is to improve the Theorem A. 

\begin{theorem}\label{nolow}
Let $s>r(1+a)+1$. Thus, the following statements are true
\begin{enumerate}
\item [1)] If $r\in (0,5/2+a)$, then the solution of IVP \eqref{gbo} is such that $u\in C([-T,T]; Z_{s,r})$.
\item [2)] If $r\in [5/2+a,7/2+a)$, then the solution of IVP \eqref{gbo} is such that $u\in C([-T,T]; \dot{Z}_{s,r})$.
\item [3)] For $r(1+a)>1/2$ and $r\in (0,5/2+a)$, the IVP \eqref{gbo} is LWP (GWP resp., if $a>1/3$) in $Z_{s,r}$. 
\item [4)] For $r(1+a)>1/2$ and $r\in [5/2+a,7/2+a)$, the IVP \eqref{gbo} is LWP (GWP resp., if $a>1/3$) in $\dot{Z}_{s,r}$. 
\end{enumerate}
\end{theorem}


Now, we put in order some remarks about our results.

\begin{enumerate}
\item [i)] In the last Theorem we push down the index $s$ given in Theorem A, in the sense that we will describe below.
From the hypotheses of Theorem A, it is seen that

\[
\llbracket (r+1)^- \rrbracket(1+a)= 
  \begin{cases}
      1+a, & \mbox{for} \ \  r\in(0,1] \\
      2(1+a), & \mbox{for} \ \ r\in (1,2].\\
       
  \end{cases}
\]

Moreover, we see that 

$$r(1+a)+1<\llbracket (r+1)^- \rrbracket(1+a)$$

is true, for $n-1<r<n-\frac{1}{1+a}$, where $n=1,...,5$.

Hence, in such intervals of $r$, we obtain less regularity compared to what appears in Theorem A.

\item [ii)] The techniques used to prove Theorem \ref{nolow} are different from those employed in \cite{FLP1}; here we use frequently the integral equation associated with the IVP \eqref{gbo} together with weighted estimates for $U(t)\phi$ (see Lemma \ref{DF} below). 

\item [iii)] Due the presence of the non linearity $uu_x$ in integral equation associated with the IVP \eqref{gbo} our Theorems \ref{n}--\ref{nolow} holds for $s\geq (1+a)r+1$. We believe that the regularity $s$ can be pushed down to $s\geq (1+a)r$; however, the techniques employed do not allow us to achieve this index.

\end{enumerate}


Now, we will describe our results about the Benjamin-Ono equation. The following theorem shows us that, assuming on initial data a decay of order five, it's possible to improve the Theorem D. This shows that the answer to the question \eqref{uniq3} is negative, even assuming more decay on initial data. 

\begin{theorem}\label{low5}

Let $u_1,u_2$ solutions of the IVP \eqref{bo} with initial data $\phi$ and $\varphi$, respectively. Suppose that conditions $\phi\neq \varphi$ and \eqref{norma}--\eqref{v1} are satisfied. In addition, assume

\begin{equation}\label{x2phi}
\int x^2\phi(x)=\int x^2\va(x)dx,
\end{equation} 
\begin{equation}\label{xphi2}
\int x\phi^2(x)=\int x\va^2(x)dx,
\end{equation}
\begin{equation}\label{I3}
I_3(\phi)=I_3(\va)
\end{equation}

and
\begin{equation}
\phi, \va \in Z_{s,5},
\end{equation}
where $s\geq 5$.

 Then $u\not = v$ and for all $T>0$
 \begin{equation}\label{u-v}
  u_1-u_2 \in C([-T,T];Z_{s,5}).
 \end{equation}

\end{theorem}


The following result improves the last Theorem. It also shows that, by assuming an additional decay on the initial data of order $3/2^{-}$, we can improve Theorem D.

\begin{theorem}\label{low5theta}

Let $\theta\in (0,1/2)$ and $u_1,u_2$ solutions of the IVP \eqref{bo} with initial data $\phi$ and $\varphi$, respectively. Suppose that the hypothesis $\phi\neq \varphi$, \eqref{norma}--\eqref{v1} and \eqref{x2phi}--\eqref{I3} are valid. In addition, assume that

\begin{equation}
\phi, \va \in Z_{s,5+\theta},
\end{equation}
where $s\geq 5+\theta$.

 Then $u\not = v$ and for all $T>0$
 \begin{equation}\label{u-v5theta}
  u_1-u_2 \in C([-T,T];Z_{s,5+\theta}).
 \end{equation}
\end{theorem}

The rest of this paper is organized as follows. In Section \ref{notation} is give some notation and the auxiliary results. The Section \ref{N} is devoted to the proof of Theorems \ref{n} and \ref{no}; here we divide the proof of Theorem \ref{no} is given in two parts, first we aboard the case $r=4$, and then the case $r=5$. In Section \ref{nolow2t} we proof the Theorems \ref{nolow} and \ref{2t}; here we used some similar estimates as in the proofs of Section \ref{N}, moreover the induction principle for the proof of Theorem \ref{nolow}. Finally, the Section \ref{NBO} is dedicated to the proof of Theorems \ref{low5} and \ref{low5theta}.

\section{Notation and Preliminaries}\label{notation}

In this section, we present some notation and auxiliary results used in our proofs. We use $c$ to denote the various positive constants that may appear in our arguments; we use a subscript to indicate the dependence on parameters. For positive numbers $a$ and $b$, we write $a \lesssim b$ if, there exists a constant $c$, such that $a\leq c b$. Let $\|\cdot\|_{L^2(\R)}$ the usual $L^2(\R)$ norm. In short, we denote the $L^2(\R)$ norm only as $\|\cdot\|$.  If $s\in \R$, we denote by $H^s:=H^s(\R)$ the  $L^2$-based Sobolev space, endowed with the norm $\|\cdot\|_{H^s}$. 
The Fourier transform of $f$ is defined by

$$\hat{f}(\xi)=\int_{\R} f(x) e^{-i \xi x}dx, \qquad \xi \in \R.$$

For any complex number $z$ and a real function $f$, the Bessel potential and Riesz potential are defined via their Fourier transforms, respectively as 

$$
\widehat{J^zf}(\xi)=(1+\xi^2)^{z/2}\hat{f}(\xi) \quad \mbox{and} \quad \widehat{D^zf}(\xi)=|\xi|^{z}\hat{f}(\xi).
$$

Let $r\in \R$. We define $L^2_{r}(\R)$ as the space all functions $f=f(x)$ satisfying
$$
\|f\|_{L^2_{r}}^2:= \int_{\R}(1+x^{2r})|f(x)|^2 dx<\infty.
$$
If $s, r\in \R$, the weighted Sobolev space is given by 

$$Z_{s,r}=H^s(\R)\cap L^2_r(\R),$$

 endowed with the norm $\|\cdot\|^2_{Z_{s,r}}=\|\cdot\|^2_{H^s}+\|\cdot\|^2_{L^2_{r}}.$

We also introduce the space of all zero mean value functions as

$$\dot{Z}_{s,r}=\{f\in Z_{s,r}:\hat{f}(0)=0\}.$$

Throughout this work, $\delta_x$ will denote the Dirac delta function, centered at $x\in \R$. 

In what follows, we describe some auxiliary results needed to obtain our main theorems.

Next, $L^{p}_{s}$ denotes the Sobolev space defined as $L^{p}_{s}:=(1-\Delta)^{-s/2}L^{p}(\R^d)$. Such spaces can be characterized by the Stein derivative of order $b$ as follows.

\begin{theorem}\label{stein}
	Let $b\in (0,1)$ and $2d/(d+2b)<p<\infty.$ Then $f\in L^{p}_{b}(\R^{d})$ if and only if
	\begin{itemize}
		\item [a)] $f\in L^{p}(\R^{d}),$ 
		\item [b)]
		$\mathcal{D}^{b}f(x):={\displaystyle \left (
			\int_{\R^{d}}\frac{|f(x)-f(y)|^{2}}{|x-y|^{d+2b}}dy\right)^{1/2}} \in
		L^{p}(\R^{d}),$ with
		\begin{equation}\label{equiv}
		\|f\|_{b,p}:=\|J^{b}f\|_{p}\simeq \|f\|_{p}+\|D^{b}f\|_{p}\simeq \|f\|_{p}+\|\mathcal{D}^{b}f\|_{p}.
		\end{equation}
	\end{itemize}
\end{theorem}
\begin{proof}
	See Theorem 1 in \cite{Stein}.
\end{proof}

The advantage in using $\mathcal{D}^{b}$ is that it is suitable when dealing with pointwise estimates, as we will see below. In addition, from  Fubini's theorem we have the product estimate (see \cite[Proposition 1]{NahasPonce})
\begin{equation}\label{Leib}
\|\mathcal{D}^{b}(fg)\|_{L^2(\R^d)} \leq \|f\mathcal{D}^{b}g\|_{L^2(\R^d)} + \|g\mathcal{D}^{b}f\|_{L^2(\R^d)}.
\end{equation}

We also recall the following.

\begin{lemma}\label{Leibnitz}
Let $b\in (0,1)$ and $h$ be a measurable function on $\R$ such that $h,h'\in L^{\infty}(\R)$. Then, for all $x\in \R$
\begin{equation}\label{Lei}
\mathcal{D}^b h(x)\lesssim \|h\|_{L^{\infty}(\R)}+\|h'\|_{L^\infty(\R)}.
\end{equation}
Moreover,
\begin{equation}\label{Leibh}
\|\mathcal{D}^{b}(h f)\|_{L^2(\R)} \leq \|\mathcal {D}^b h\|_{L^\infty(\R)} \|f\|_{L^2(\R)} + \|h\|_{L^\infty(\R)} \|\mathcal{D}^{b}f\|_{L^2(\R)}.
\end{equation}
\end{lemma}
\begin{proof}
Note that \eqref{Lei} follows from Lemma 2.7 in  \cite{pastran}, while \eqref{Leibh} is a consequence of \eqref{Leib}.
\end{proof}

In our proofs, we need to deal with several pointwise estimates in terms of the Stein derivative. To handly with this kind estimates, we start by introducing a cut-off function 
\begin{equation}\label{chi}
\chi\in C_0^\infty(\R) \ \mbox{such that}\ \mbox{supp}\ \chi\subset [-2,2] \ \mbox{ and} \ \chi\equiv1 \ \mbox{in}  \ (-1,1).
\end{equation}

The next result will be needed in the proof of Lemma \ref{DF1}.

\begin{proposition}\label{Dstein}
For any $\theta \in (0,1)$ and $\alpha >0,$ the function $\mathcal{D}^\theta (|\xi|^\alpha \chi (\xi))(\cdot)$ is continuous in $\eta \in \R-\{0\}$  with
$$\mathcal{D}^\theta (|\xi|^\alpha \chi(\xi))(\eta) \sim \left\{\begin{array}{lcc}
c|\eta|^{\alpha -\theta}+c_1,& \quad \alpha \not= \theta, |\eta|\ll 1, \\
c(-\ln |\eta|)^{1/2}, & \quad \alpha=\theta, |\eta|\ll 1,\\
\frac{c}{|\eta|^{1/2+\theta}}, & \quad  |\eta|\gg 1,
\end{array}\right.
$$
 in particular, one has that
\begin{equation}\label{Dstein4}
\mathcal{D}^\theta (|\xi|^\alpha \chi (\xi))\in L^{2}(\R) \ \mbox{if and only if} \ \theta< \alpha +1/2.
\end{equation}
In a similar fashion
\begin{equation}\label{Dstein1}
\mathcal{D}^\theta (|\xi|^\alpha \sgn(\xi) \chi (\xi))\in L^{2}(\R) \ \mbox{if and only if} \ \theta< \alpha +1/2.
\end{equation}
\end{proposition}
\begin{proof}
See Proposition 2.9 in \cite{FLP1}.
\end{proof}


For the proof of Theorems \ref{no} and \ref{2t}, the key ingredient is the following.

\begin{proposition}\label{DsteinL3}
If $\gamma \in (0,1/2)$ and $0<\epsilon<\gamma$ then
\begin{equation}\label{Dstein3}
\mathcal{D}^{\gamma-\epsilon} (|\xi|^{\gamma-1/2}\chi(\xi))\in L^{2}(\R).
\end{equation}
\end{proposition}
\begin{proof}
	See Proposition 2.9 in \cite{ap1}.
\end{proof}

The next result is necessary to obtain the Lemma \ref{DF}.

\begin{lemma}\label{Pontual1}
Let $b\in (0,1)$. For any $t>0$,
\begin{equation*}
\mathcal{D}^{b}(e^{-itx|x|^{1+a}})\les t^{b/(2+a)}+t^{b}|x|^{(1+a)b}.
\end{equation*}
\end{lemma}
\begin{proof}
See Proposition 2.7 in \cite{FLP1}.
\end{proof}

\begin{lemma}\label{interx}
Let $\alpha,b>0.$ Assume that $J^{\alpha}f\in L^{2}(\R)$ and
$\langle x \rangle^b f=(1+x^2)^{b/2}f(x)\in L^{2}(\R).$ Then, for any
$\beta \in (0,1)$,
\begin{equation}\label{inter1x}
\|J^{\alpha \beta}(\langle x \rangle^{(1-\beta)b}f)\|_{L^2}\leq c\|\langle  x
\rangle^{b}f\|_{L^2}^{1-\beta}\|J^{\alpha}f\|_{L^2}^{\beta}.
\end{equation}
\end{lemma}
\begin{proof}
See Lemma 4 in \cite{NahasPonce}. 
\end{proof}

The next result is useful to prove the Lemma \ref{DF1} below.

\begin{proposition}\label{Jota}
Let  $\varrho \in L^{\infty}(\R)$, with $\partial_x^k \varrho \in L^{2}(\R)$ for $k=1,2$. Then, for any $\theta\in (0,1)$, there exists a constant $c>0$, depending only on $\varrho$ and $\theta$, such that

\begin{equation}\label{Jotaf1}
\|J^\theta(\varrho f)\|_{L^2}\leq c \|J^\theta f\|_{L^2}.
\end{equation}
\end{proposition}
\begin{proof}
See Propositions 2.4 and 2.5 in \cite{FLP1}.
\end{proof}

The integral equation associated to the IVP \eqref{gbo} is given by
\begin{equation}\label{inte}
u(t)=U(t)\phi-\int_0^t U(t-\tau)\vartheta(\tau)d\tau,
\end{equation} 
where $\vartheta(\tau):=\frac{1}{2}\p_x u^2$ and $U(t)\phi:=(e^{-i\xi|\xi|^{a+1}}\hat{\phi}(\xi))^{\vee}$.

The arguments of our proofs are based on derivatives with respect to $\xi$-variable of the function
\begin{equation}\label{psidef}
 \psi(\xi,t):=e^{-it\xi|\xi|^{1+a}}.
\end{equation} 

Hence, after several computations we conclude 

\begin{equation}
\begin{split}\label{F2} 
\p_\xi(\psi \hat{f})=&\;\psi \Big \{-it(2+a)|\xi|^{1+a}\hat f+\p_\xi \hat f\Big\},
\end{split}
\end{equation}

\begin{equation}
\begin{split}\label{F3} 
\p_\xi^{2}(\psi \hat{f})=&\;\psi \Big \{\big[-it(2+a)(1+a)|\xi|^a\s-t^2(2+a)^2|\xi|^{2(1+a)}\big]\hat{f}+\\
&-2it(2+a)|\xi|^{1+a}\p_\xi \hat{f}+\p_\xi^2 \hat f\Big\},
\end{split}
\end{equation}

\begin{equation}
\begin{split}\label{F4} 
\p_\xi^{3}(\psi \hat{f})=&\;\psi \Big \{(2+a)t\big[-ia(1+a)|\xi|^{a-1}-3t(2+a)(1+a)|\xi|^{2a+1}\s+i(2+a)^2 t^2|\xi|^{3(1+a)}\big]\hat{f}+\\
&+3\big[-3it(2+a)(1+a)|\xi|^{a}\s-3t^2(2+a)^2|\xi|^{2(1+a)}\big]\p_\xi \hat{f}+\\
&+3\big[-3it(2+a)|\xi|^{1+a}\big]\p_\xi^2 \hat{f}+\p_\xi^3 \hat f\Big\},
\end{split}
\end{equation}

\begin{equation}
\begin{split}\label{F5} 
\p_\xi^{4}(\psi \hat{f})=&\;\psi \Big \{\big[-2(2+a)(a^2-1)a i t|\xi|^{a-2}\s-(2+a)^2(1+a)(3+7a)t^2|\xi|^{2a}+\\
&+6i(1+a)(2+a)^3 t^3\si|\xi|^{2+3a}+t^4 (2+a)^4 |\xi|^{4(1+a)}\big]\hat{f}+4\big[it^3 (2+a)^3 |\xi|^{3(1+a)}\\
&-(2+a)(1+a)ait|\xi|^{a-1}-3t^2(2+a)^2(1+a)\sgn(\xi)|\xi|^{1+2a}\big]\p_\xi \hat{f}\\
&+6\big[-it(2+a)(1+a)\sgn(\xi)|\xi|^a  -t^2 (2+a)^2 |\xi|^{2(1+a)}\big]\p_\xi^2 \hat{f} \\
&-4it(2+a)|\xi|^{1+a}\p_\xi^3 \hat{f} +\p_\xi^4 \hat{f}\Big\}\\
=:&\;A_1+ \cdots + A_{11},
\end{split}
\end{equation}
and
\begin{equation}
\begin{split}\label{F6} 
\p_\xi^{5}(\psi\hat{f})=&\;\psi \Big \{\big[-(a^2-4)(a^2-1)ait|\xi|^{a-3}-5a(2+a)^2(1+a)(3a+1)t^2\s|\xi|^{2a-1}+\\
&+(1+a)\big(6(2+3a)+(2+a)^3\big)it^3|\xi|^{1+3a}+4t^4 (a+1)(2+a)^4 \s|\xi|^{4a+3}+\\
&+6(2+a)^4(1+a)t^4\s|\xi|^{4a+2}-(2+a)^5it^5|\xi|^{5(a+1)}\big]\hat{f}+\\
&+5\big[-2(2+a)(a^2-1)ait\s |\xi|^{a-2}-(2+a)^2(1+a)(3+7a)t^2|\xi|^{2a}+\\
&+6i(1+a)(2+a)^3 t^3\s|\xi|^{2+3a}+t^4 (2+a)^4|\xi|^{4(a+1)}\big]\p_\xi \hat{f}\\
&+10\big[-(2+a)(1+a)ait|\xi|^{a-1}  -3t^2(2+a)^2(1+a)\s |\xi|^{2a+1}+\\
&+it^3 (2+a)^3|\xi|^{3(a+1)}\big]\p_\xi^2 \hat{f}+10\big[-it(2+a)(1+a)\s|\xi|^a-t^2 (2+a)^2|\xi|^{2(1+a)}\big]\p_\xi^3 \hat{f}\\
&-5it(2+a)|\xi|^{1+a}\p_\xi^4 \hat{f}+\p_\xi^5 \hat{f}\Big\}\\
=:&\;B_1+ \cdots + B_{18}.
\end{split}
\end{equation}

Note that $A_j$ and $B_k$ depends on $\xi,t$ and $\hat{f}$, that is, $A_j=A_j(\xi,t,\hat{f})$ and $B_k=B_k(\xi,t,\hat{f})$, where $j=1,...,11$ and $k=1,...,18$.

The next three results below are useful to prove the Theorems \ref{no}--\ref{nolow}.

\begin{lemma}\label{DF}Let $\psi$ be as in \eqref{psidef}. For all $\theta \in (0,1)$ and $t>0$,
\begin{equation*}
\|\Dt(\psi \hat{f})\|\lesssim \rho(t)\big(\|J^{(1+a)\theta}f\|+\||x|^\theta f\|\big),
\end{equation*}
where $\rho(t)=1+t^\theta+t^{\frac{\theta}{2+a}}$.
\end{lemma}
\begin{proof}

The proof follows from Lemma \ref{Pontual1} and is similar to that of Lemma 2.13 in \cite{ap1}; so we omit the details.

\end{proof}

\begin{lemma}\label{DF1} Let $\theta, a \in (0,1)$ and $\beta\geq 0$. Then
\begin{equation}\label{DF11}
\||x|^\theta D^{a+\beta}f\|\lesssim \|J^{(a+\beta)(1-\frac{\theta}{r})^{-1}}f\|+\||x|^{r} f\|,
\end{equation}
where $r\geq 1$ and $\theta<a+1/2$. 

This result still holds if, in the left-hand side of the above inequality, we replace $f$ by $\h f$.
\end{lemma}
\begin{proof}
	We point out that here and throughout the paper we will use Plancherel's identity without mentioning.
Let $\chi$ be as in \eqref{chi}. Then,
\begin{equation}
\begin{split}\label{d1}
\||x|^\theta D^{a+\beta}f\|&=\|D_\xi^\theta(|\xi|^{a+\beta}\hat f)\|\\
&\les \|\D(|\xi|^{a+\beta}\chi\hat f)\|+\|\D(|\xi|^{a+\beta}(1-\chi)\hat f)\|\\
                           &:=\bar B_1+\bar B_2.
\end{split}
\end{equation}
Since $\theta<a+1/2\leq a+\beta+1/2$, from  \eqref{Leib} and \eqref{Dstein4} it follows that 
\begin{equation}
\begin{split}\label{d2}
\bar B_1 \les \||\xi|^{a+\beta}\chi \|_{L^\infty}\|\Dt\hat f\|+\|\Dt(|\xi|^{a+\beta}\chi)\|\|\hat f\|_{L^\infty}\les \||x|^\theta f\|+\|\lanx f\|,
\end{split}
\end{equation}
where  we also used the Sobolev embedding.

For $\bar B_2$ we first observe that
$$
\bar B_2\lesssim \Big\|J_\xi^\theta\Big(\dfrac{|\xi|^{a+\beta}(1-\chi)}{\lan^{a+\beta}}\lan^{a+\beta}\hat f\Big)\Big\|.
$$
A straightforward calculation reveals the function
$$
\varrho(\xi):=\dfrac{|\xi|^{a+\beta}(1-\chi)}{\lan^{a+\beta}}
$$
satisfies the assumptions of Proposition \ref{Jota}. Thus, an application of Lemma \ref{interx} gives
\begin{equation}
\begin{split}\label{d3}
\bar B_2 &\leq \|J_\xi^\theta(\lan^{a+\beta}\hat f)\|\les \|J_\xi^r \hat f\|+\|\lan^{(a+\beta)(1-\frac{\theta}{r})^{-1}}\|\les \|\lanx^r f\|+\|J^{(a+\beta)(1-\frac{\theta}{r})^{-1}}f\|.
\end{split}
\end{equation}
Gathering together \eqref{d1}--\eqref{d3}, we obtain the result.

The proof with $\h f$ instead of $f$ on the left-hand side of \eqref{DF11}, follows in a similar way.
This finishes the proof.
\end{proof}

\begin{lemma}\label{DF2} Let $a \in (0,1)$ and $\beta\geq 0$. Then
\begin{equation*}
\|J^{(a+\beta)\mu}(x^{j}f)\|\lesssim \|J^{(a+\beta)(j+\mu)}f\|+\||x|^{j+\mu} f\|,
\end{equation*}
where $\mu>0$ and $j\in \{1,...,4\}$.
\end{lemma}
\begin{proof}
Assuming $j=1$ we deduce 
\begin{equation*}
\begin{split}
\|J^{(a+\beta)\mu}(xf)\|&=\|\lan^{(a+\beta)\mu}\p_\xi \hat f\|\\
                        &\les \|\lan^{(a+\beta)\mu-1}\hat f\|+\|J_\xi(\lan^{(a+\beta)\mu}\hat f)\|\\
                        &\les \|J^{(a+\beta)\mu-1}f\|+\|J_\xi^{1+\mu}\hat f\|+\|\lan^{(1+\mu)(a+\beta)}\hat f\|\\
                        &\les \|J^{(a+\beta)(1+\mu)}f\|+\||x|^{1+\mu}f\|, 
\end{split}
\end{equation*}
where we used Lemma \ref{interx}. 

For $j\in \{2,3,4\}$ the idea of the proof is the same above.

This ends the proof.
\end{proof}

In what follows, we will give the auxiliary results necessary to obtain our results for the Benjamin-Ono equation. We start by recalling the definition of the truncated weights $\langle x\rangle_N$, $N\in \Z^{+}$, which are given by
\begin{eqnarray*}
\langle x\rangle_N:=\left\{\begin{array} {lccc}
\langle x \rangle \ \mathrm{if} \  |x|\leq N,\\
2N \ \mathrm{if} \ |x|\geq 3N,
\end{array} \right.
\end{eqnarray*}
where $\langle x \rangle = (1+x^2)^{1/2}$. Also, we assume that $\langle x\rangle_N$
is smooth and non-decreasing in $|x|$ with $\langle x\rangle_N'(x)\leq 1,$ for any
$x\geq 0$, and there exists a constant $c$ independent of $N$ such that
$|\langle x\rangle_N''(x)|\leq c \partial_x^{2}\langle x \rangle.$

To obtain the proof of Theorem \ref{low5}, we need of the following auxiliary result. The Proposition 1 in \cite{flores} tells us that 

\begin{equation}\label{Ic}
\frac{d}{dt}\int x u^2(x,t)dx=2I_3(u)
\end{equation}
and
\begin{equation}\label{IIc}
\frac{d}{dt}\int x^2 u(x,t)dx=\int x u^2(x,t)dx,
\end{equation}
where $u=u(x,t)$ is solution of the IVP \eqref{bo} and $I_3$ is given in \eqref{I3}.

To obtain the proof of Theorem \ref{low5theta} we need of the following.
\begin{remark}\label{remark2}
Let $u_1$ and $u_2$ solutions of the IVP \eqref{bo}, with initial data $\phi$ and $\va$, respectively.
If $w=u_1-u_2$, then

\begin{equation}\label{x2w}
\int x^2 w(x,t)dx=t^2(I_3(\phi)-I_3(\va))+t\int x(\phi^2(x)-\va^2(x))dx+\int x^2 (\phi(x)-\va(x))dx,
\end{equation} 

for all $t$ in which the solutions $u_1$ and $u_2$ there exist.
\end{remark}
\begin{proof}
The identity \eqref{IIc} implies

\begin{equation*}
\frac{d}{dt}\int x^2 w(x,t)dx=\int x \big (u_1^2 (x,t) -u_2^2 (x,t)\big)dx.
\end{equation*}

Taking the derivative in the last equality, using \eqref{Ic} and the conservation law \eqref{I3} we conclude that

\begin{equation}\label{d2I3}
\frac{d^2}{dt^2}\int x^2 w(x,t)dx=2t(I_3(\phi)-I_3(\varphi)).
\end{equation} 
Hence, identities \eqref{Ic} and \eqref{d2I3} yields us the desired result. 

This ends the proof.
\end{proof}

\section{Proof of Theorems \ref{n}--\ref{no}} \label{N}

The results that we give in this section shows us that the kind of uniqueness properties for the KdV equation (\cite{EKPV}) and Schr\"odinger equation (\cite{EKPV1}) doesn't valid for the DGBO equation. As we already mentioned in the introduction, the objective of the Theorems \ref{n} and \ref{no} is to get a negative answer to the question \eqref{uniq3}.

\begin{proof}[Proof of Theorem \ref{n}]
First we observe that	since $5/2+a^{-}<4$ we have  $\phi,\varphi\in Z_{s,5/2+a^{-}}$. Hence, 
from Theorem A (part (iii)) there exists $T>0$ such that $u_1,u_2 \in C([-T,T];Z_{s,5/2+a^{-}})$. Consequently, the constant
$$N:=\sup_{[-T,T]}\{\|u_1(t)\|_{H^{4(1+a)+1}}+\|u_2(t)\|_{H^{4(1+a)+1}}+\||x|^{\frac52}u_1(t)\|+\||x|^{\frac52}u_2(t)\|\},$$
is finite. 
 
By multiplying \eqref{int} by $x^4$ and using the Fourier transform we obtain
\begin{equation}\label{intxi1}
\begin{split}
\p_\xi^4\widehat{w(t)}=& \p_\xi^4(\psi(\xi,t)\ha)-\int_0^t  \p_\xi^4(\psi(\xi,t-\tau)\hat{z}(\tau))d\tau.
\end{split}
\end{equation} 
The idea now is to show that the right-hand side of \eqref{intxi1} is finite. Using \eqref{F5} with $\sigma$ instead of $f$, the first term on the right-hand side of \eqref{intxi1} can be estimated as 
\begin{equation}
\begin{split}\label{x4}
\|\p_\xi^4(\psi(\xi,t)\ha)\|\les & \ \||\xi|^{a-2} \ha\|+\||\xi|^{2a}\ha\|+\||\xi|^{3a+2}\ha\|++\||\xi|^{4(a+1)} \ha\|+\||\xi|^{3(a+1)} \p_\xi \ha\|\\
&+\||\xi|^{a-1}\p_\xi \ha\|++\||\xi|^{2a+1}\p_\xi \ha\|+\||\xi|^{a}\p_\xi^2\ha\|+\||\xi|^{2(1+a)} \p_\xi^2 \ha\|+\\
&+\||\xi|^{1+a} \p_\xi^3 \ha\|+\|\p_\xi^4 \ha\|\\
=:& \ C_1+\cdots+C_{11}.
\end{split}
\end{equation}

Let us estimate some of the terms $C_j$ above. Note that \eqref{v} and \eqref{v1} are equivalent to $\ha(0)=0$ and $ \p_\xi \ha(0)=0$, respectively. So, from Taylor's expansion with integral remainder we may write
\begin{equation}\label{taylor0}
\hat{\sigma}(\xi)=\int_0^\xi (\xi-\zeta)\partial_\xi ^2 \hat{\sigma}(\zeta)d\zeta.
\end{equation}
Using \eqref{taylor0} and the Sobolev embedding, the term $C_1$ can be estimated as
\begin{equation}\label{C1}
\begin{split}
C_1 \leq& \ \|\chi|\xi|^{a-2} \int_0^\xi (\xi-\zeta)\partial_\xi ^2 \hat{\sigma}(\zeta)d\zeta\|+\|(1-\chi)|\xi|^{a-2} \ha\|\\
    \les& \ \|\p_\xi^2 \ha\|_{L^\infty}\|\chi|\xi|^{a-2}\xi^2\|+\|(1-\chi)|\xi|^{a-2}\|_{L^\infty}\| \ha\|\\
    \les& \ \|\lanx^{5/2+a^{-}}\sigma\|+\|\sigma\|.
\end{split}
\end{equation}
Note that in view of the support of $\chi$, the quantities $\|\chi|\xi|^{a-2}\xi^2\|$ and $\|(1-\chi)|\xi|^{a-2}\|_{L^\infty}$ are finite.  Also, the fundamental theorem of calculus and hypothesis \eqref{v1} implies that
\begin{equation}\label{tfc}
\p_\xi \ha (\xi)=\int_0^\xi \p_\xi^2 \ha (\zeta)d\zeta.
\end{equation} 
Hence, in view of \eqref{tfc},
\begin{equation}
\begin{split}
C_6 \leq& \ \|\chi|\xi|^{a-1} \p_\xi \ha\|+\|(1-\chi) |\xi|^{a-1} \p_\xi \ha\|\\
    \les& \ \|\p_\xi^2 \ha\|_{L^\infty}\|\chi|\xi|^{a-1}|\xi|\|+\|(1-\chi)|\xi|^{a-1}\|_{L^\infty}\| \p_\xi\ha\|\\
    \les & \ \|\lanx^{5/2+a^{-}}\sigma\|+\|\p_\xi\ha\|\\
    \les & \ \|\lanx^{5/2+a^{-}}\sigma\|.
\end{split}
\end{equation}
Using the inequality $|\langle\xi\rangle^{r}\p_\xi\ha|\lesssim |\p_\xi(\langle\xi\rangle^{r}\ha)|+|\langle\xi\rangle^{r-1}\ha|$, an application of Lemma \ref{interx} gives
\begin{equation}\label{C5}
\begin{split}
C_5 \leq \|J_\xi(\lan^{3(a+1)}\ha)\|+\|\lan^{3a+2}\ha\|\les \|J_\xi^4 \ha\|+\|\lan^{4(1+a)}\ha\|\les \|\lanx^4 \sigma\|+\|\sigma\|_{H^{4(1+a)}}.
  \end{split}
\end{equation}
The remaining terms $C_j$ may be estimated as in \eqref{C5}, so we omit the details. Therefore, we deduce
\begin{equation}\label{4sigma}
\|\p_\xi^4(\psi(\xi,t)\ha)\|\les \ \|\lanx^4 \sigma\|+\|\sigma\|_{H^{4(1+a)}}.
\end{equation}

To deal with the second term on the right-hand side of \eqref{intxi1} we observe that
\begin{equation}\label{lz}
\hat z(0,\tau)=0, \quad \tau \in [-T,T].
\end{equation}
In addition, by \eqref{norma} and the conservation law \eqref{norm},
\begin{equation}\label{l1z}
\begin{split}
\p_\xi \hat z(0,\tau)=\frac{i}{2}\big(\widehat{u_1^2}-\widehat{u_2^2}\big)(0,\tau)=\frac{i}{2}\int (u_1^2(x,\tau)-u_2^2(x,\tau))dx=\frac{i}{2}(\|\phi\|^2-\|\varphi\|^2)=0.
\end{split}
\end{equation}

In view of \eqref{lz} and \eqref{l1z} we may obtain a similar estimate as in \eqref{4sigma} but with   $\psi(\xi,t-\tau)\hat{z}(\tau)$ instead of $\psi(\xi,t)\ha$. Hence,
\begin{equation}
\begin{split}\label{x4uz}
\|\p_\xi^4(\psi(\xi,t-\tau)\hat{z}(\tau))\|&\les \ \|J^{4(1+a)}z(\tau)\|+\|x^4 z(\tau)\|\\
&\les \ \|J^{4(1+a)}(\p_x w(u_1+u_2))\|+\|x^{4}\p_x w(u_1+u_2)\|+\\
&\quad+ \|J^{4(1+a)}(w\p_x(u_1+u_2))\|+\|x^{4}w\p_x(u_1+u_2)\|\\
&=: G_1+\cdots+ G_4.
\end{split}
\end{equation}
For $G_1$, we use that $H^{4(1+a)}$ is Banach algebra to deduce
\begin{equation}
\begin{split}\label{G1}
G_1&\les \|J^{4(1+a)}\p_x w\|\|J^{4(1+a)}(u_1+u_2)\|\\
&\les \|J^{4(1+a)+1} w\|(\|J^{4(1+a)}u_1\|+\|J^{4(1+a)}u_2\|)\\
&\les N^2.
\end{split}
\end{equation}
Note that at this point we have indeed used that $s\geq 4(1+a)+1$. In a similar fashion,
\begin{equation}
\begin{split}\label{G3}
G_3&\les N^2.
\end{split}
\end{equation}
In addition, Sobolev's embedding and Lemma \ref{interx} (with $b=5/2$ and $\alpha=5$) imply
\begin{equation}
\begin{split}\label{G2}
G_2&\les \ \|x^2 (u_1+u_2)\|_{L^\infty_x}\|x^2\p_x w\|\\
   &\les \ \|J(\lanx^2 (u_1+u_2))\|\left(\|J(\lanx^2 w)\|+\|\langle x\rangle w\|\right)\\
   &\les \ \left(\|J^5 (u_1+u_2)\|+\|\lanx^{5/2}(u_1+u_2)\|\right)\left(\|J^5 w\|+\|\lanx^{5/2}w\|\right)\\
   &\les \ N^2
\end{split}
\end{equation}
and, similarly
\begin{equation}
\begin{split}\label{G4}
G_4\les \ N^2.
\end{split}
\end{equation}
Thus, by \eqref{x4uz}--\eqref{G4} we obtain
\begin{equation}
\begin{split}\label{X4U}
\|\p_\xi^4(\psi(\xi,t-\tau)\hat{z}(\tau))\|\les N^2,  \quad \tau \in [-T,T].
\end{split}
\end{equation}
Therefore, using \eqref{intxi1}, \eqref{4sigma} and \eqref{X4U} we conclude that for all $t\in [-T,T]$
\begin{equation*}
\begin{split}
\|x^4 w(t)\|&\les \ \|\lanx^4 \sigma\|+\|\sigma\|_{H^{4(1+a)}}+\int_0^t N^2d\tau\\
            &\les  \ \|\lanx^4 \phi\|+\|\lanx^4 \varphi\|+\|\phi\|_{H^{4(1+a)}}+\|\varphi\|_{H^{4(1+a)}}+tN^2.   
\end{split}
\end{equation*}
All terms on the right-hand side above are finite in view of our assumptions. 

The continuity of $t\in [-T,T]\mapsto w(t)\in Z_{s,4}$ run as in \cite{AP}.

This finishes the proof of the theorem.
\end{proof}

For the proof of the next Theorem, some of the key ingredients used are the Lemma \ref{DF1} and the Proposition \ref{DsteinL3}. Such Theorem shows that even assuming a decay of order $11/2^{-}$, we still obtain a negative answer to question \eqref{uniq3}.

\begin{proof}[Proof of Theorem \ref{no}]

First, we will deal with the proof of case $r=4$. Let $u_1$ and $u_2$ be the solutions of the IVP \eqref{gbo} with initial data $\phi$ and $\varphi$, respectively. Theorem A again allows to obtain $T>0$ such that $u_1,u_2 \in C([-T,T];Z_{s,5/2+a^{-}})$. First, we note that in view of $\theta<1/2+a$ it follows that
\begin{equation}\label{nu}
\nu:=2+\theta<5/2+a.
\end{equation}
Also, using Theorem \ref{n} we obtain $w=u_1-u_2\in C([-T,T];Z_{s,4})$, which implies that the constant
 $$M:=\sup_{[-T,T]}\{\|u_1(t)\|_{H^{s}}+\|u_2(t)\|_{H^{s}}+\||x|^{\nu}u_1(t)\|+\||x|^{\nu}u_2(t)\|+\|x^4 w(t)\|\},$$
 is finite.


Multiplying \eqref{int} by $|x|^{4+\theta}$ and using Plancherel's identity we get
\begin{equation}\label{intxi}
\begin{split}
D_\xi^\theta \p_\xi^4(\widehat{w(t)})=&D_\xi^\theta \p_\xi^4(\psi(\xi,t)\ha)-\int_0^t D_\xi^\theta \p_\xi^4(\psi(\xi,t-\tau)\hat{z}(\tau))d\tau.
\end{split}
\end{equation} 
In order to estimate the first term on the right-hand side of \eqref{intxi}, we use \eqref{F5} (with $\sigma$ instead of $f$)  to write
\begin{equation}
\begin{split}\label{dtheta}
\|D_\xi^\theta \p_\xi^4(\psi(\xi,t)\ha)\|\les\, & \|\D(\psi \s |\xi|^{a-2} \ha)\|+\|\D(\psi |\xi|^{2a}\ha)\|+\|\D(\psi |\xi|^{3a+2} \s\ha)\|+\\
&+\|\D(\psi |\xi|^{4(a+1)} \ha)\|+\|\D(\psi |\xi|^{3(a+1)} \p_\xi \ha)\|+\|\D(\psi |\xi|^{a-1}\p_\xi \ha)\|+\\
&+\|\D(\psi |\xi|^{2a+1} \s \p_\xi \ha)\|+\|\D(\psi |\xi|^{a}\s \p_\xi^2\ha)\|+\|\D(\psi |\xi|^{2(1+a)} \p_\xi^2 \ha)\|+\\
&+\|\D(\psi |\xi|^{1+a} \p_\xi^3 \ha)\|+\|\D(\psi \p_\xi^4 \ha)\|\\
=:& \ E_1+\cdots+E_{11}.
\end{split}
\end{equation}
First, we will deal with terms $E_1$ and $E_6$, that present more difficulties. From Lemma \ref{DF} we deduce
\begin{equation}
\begin{split}\label{splite1}
E_1 \les \|J^{(1+a)\theta}D^{a-2}\h\sigma\|+\||x|^{\theta}D^{a-2}\h \sigma\|:=E_{1,1}+E_{1,2}.
\end{split}
\end{equation}
Let us write
\begin{equation}
\begin{split}\label{splite2}
E_{1,2} \les \|D^{\theta}_\xi(\s |\xi|^{a-2}\chi \ha)\|+\|\D(\s |\xi|^{a-2}(1-\chi)\ha)\|:=E_{1,2}^1+E_{1,2}^2.
\end{split}
\end{equation}
 Using the assumption $\hat{\sigma}(0)=\p_\xi \hat{\sigma}(0)=0$ and Taylor's expansion with integral remainder we obtain
\begin{equation}\label{taylor}
\hat{\sigma}(\xi)=\int_0^\xi (\xi-\zeta)\partial_\xi ^2 \hat{\sigma}(\zeta)d\zeta.
\end{equation}
Thus,
\begin{equation}
\begin{split}\label{E121est}
E_{1,2}^1&=\Big\|D^{\theta}_\xi\big(|\xi|^{a-2}\s \chi \int_0^\xi (\xi-\zeta)\partial_\xi ^2 \hat{\sigma}(\zeta)d\zeta \big)\Big\|\\
&\les \Big\|D^{\theta}_\xi \big(|\xi|^{a-2}\s \chi \int_0^\xi (\xi-\zeta)(\partial_\xi ^2 \hat{\sigma}(\zeta)-\p_\xi^2 \hat{\sigma}(0))d\zeta \big)\Big\|+\\
&\quad+\Big\|D^{\theta}_\xi(|\xi|^{a-2}\s \chi \int_0^\xi (\xi-\zeta)\p_\xi^2 \hat{\sigma}(0)d\zeta )\Big\|\\
&=\Big\|D^{\theta}_\xi \big(\underbrace{|\xi|^{a-2}\s \chi \int_0^\xi (\xi-\zeta)\int_0^\zeta \p_\xi^3 \hat{\sigma}(\eta)d\eta d\zeta}_{L} \big)\Big\|+\frac12\Big\|D^{\theta}_\xi(|\xi|^{a-2}\s \chi \xi^2 \p_\xi^2 \hat{\sigma}(0))\Big\|\\
&\les  \|\D L\|+\|\p_\xi^2 \ha\|_{L^\infty_\xi}\|\D (|\xi|^a \s \chi)\|.
\end{split}
\end{equation}
To estimate the first term, we will use the inequality $ \|\D L\|\lesssim \|L\|^\theta\|\p_\xi L\|^{1-\theta}$. Now, using Sobolev embedding we get
\begin{equation*}
\begin{split}
\|L\|&\les \Big \| |\xi|^{a-2}\chi \int_0^\xi (\xi-\zeta)\int_0^\zeta|\p_\xi^3 \hat{\sigma}(\eta)|d\eta d\zeta\Big\|\\
&\les \|\p_\xi^3 \ha\|_{L^\infty_\xi}\||\xi|^{a-2}\chi |\xi|^3\|\\
&\les \|\lanx^{4+\theta} \sigma\|\| |\xi|^{a+1}\chi\|\\
&\les \|\lanx^{4+\theta} \sigma\|
\end{split}
\end{equation*}
and
\begin{equation*}
\begin{split}
\|\p_\xi L\|\les& \Big \| |\xi|^{a-3}\chi \int_0^\xi (\xi-\zeta)\int_0^\zeta\p_\xi^3 \hat{\sigma}(\eta)d\eta d\zeta\Big\|+ \Big \| |\xi|^{a-2}\p_\xi\chi \int_0^\xi (\xi-\zeta)\int_0^\zeta\p_\xi^3 \hat{\sigma}(\eta)d\eta d\zeta\Big\|\\
&+\Big \| |\xi|^{a-2}\chi \int_0^\xi \int_0^\zeta\p_\xi^3 \hat{\sigma}(\eta)d\eta d\zeta\Big\|\\\\
&\les \|\p_\xi^3 \ha\|_{L^\infty_\xi}\||\xi|^{a-3}\chi |\xi|^3\|+\|\p_\xi^3 \ha\|_{L^\infty_\xi}\||\xi|^{a-2}\p_\xi \chi |\xi|^3\|+\|\p_\xi^3 \ha\|_{L^\infty_\xi}\||\xi|^{a-2}\chi \xi^2\|\\
&\les \|\lanx^{4+\theta} \sigma\|,
\end{split}
\end{equation*}
where  we used that $\p_\xi \s=\delta_0$. The last two inequalities imply that $\|\D L\|\lesssim \|\lanx^{4+\theta} \sigma\|$.

For the second term on the right-hand side of \eqref{splite2}, we use interpolation to obtain
\begin{equation*} 
\begin{split}
E_{1,2}^2 &\les \||\xi|^{a-2}(1-\chi)\ha\|+\|\p_\xi \big(\sgn(\xi)|\xi|^{a-2}(1-\chi)\ha\big)\|\\
&\les \Big\|\frac{1-\chi}{\xi^2} \Big\|_{L^\infty}\||\xi|^a \ha\|+\Big\|\frac{1-\chi}{\xi^3}\Big\|_{L^\infty} \||\xi|^a\ha\|+\\
&\quad+\||\xi|^{a-2}\p_\xi \chi\|_{L^\infty}\| \ha\|+\Big\|\frac{1-\chi}{\xi^2}\Big\|_{L^\infty} \||\xi|^a\p_\xi\ha\|\\
&\les \|\sigma\|+\|D^a \sigma\|+\|D^a(x\sigma)\|\\
&\les \|J^2 \sigma\|+\|\lanx^2 \sigma\|.
\end{split}
\end{equation*}

In view of $\theta,a\in (0,1)$, we also have 
\begin{equation}
\begin{split}
E_{1,1}\les \|J^2 D^{a-2}\sigma\|\leq \|D^a \sigma\|+\|D^{a-2}\sigma\|\les \|J^a \sigma\|+\|D^{a-2}\sigma\|,
\end{split}
\end{equation}
where the term $\|D^{a-2}\sigma\|$ can be estimated as similar way to $E_{1,2}$.

Hence, by the above estimates

\begin{equation}\label{E1}
E_1\les \|J^2\sigma\|+\|\lanx^{4+\theta} \sigma\|.
\end{equation}

To estimate the term $E_6$ we can use Lemma \ref{DF} to obtain 
\begin{equation}
\begin{split}\label{E6plit}
E_{6}\les \|J^{(1+a)\theta} D^{a-1}(x\sigma)\|+\||x|^\theta D^{a-1}(x\sigma)\|:=E_{6,1}+E_{6,2}.
\end{split}
\end{equation}
In view of $\p_\xi \ha(0)=0$, the Taylor's theorem with integral remainder implies that

\begin{equation}\label{taylor2}
\p_\xi \ha(\xi)=\xi \p_\xi^2 \ha(0)+\int_0^\xi (\xi-\zeta)\p_\xi^3 \ha(\zeta)d\zeta.
\end{equation}
Thus
\begin{equation}
\begin{split}\label{estE622}
E_{6,2}&= \|\D(|\xi|^{a-1}\chi \p_\xi \ha)\|+\|\D(|\xi|^{a-1}(1-\chi) \p_\xi \ha)\|:=E_{6,2}^1+E_{6,2}^2,
\end{split}
\end{equation}
where by using \eqref{taylor2}
\begin{equation}
\begin{split}
E_{6,2}^1&\les \|\D(|\xi|^{a}\s\chi \p_\xi^2 \ha(0))\|+\Big\|\D\big(\underbrace{|\xi|^{a-1}\chi \int_0^\xi (\xi-\zeta)\p_\xi^3 \ha(\zeta)d\zeta}_{R}\big)\Big\|\\
&\les \li \|\Dta(|\xi|^a \s \chi)\|+\|\D R\|.
\end{split}
\end{equation}
The term $R \in H^1(\R)$, in fact

\begin{equation}
\begin{split}
\|R\|&\les \||\xi|^{a-1}\chi \xi^2\|\lii\\
&\les \||\xi|^{a+1}\chi\|\|\lanx^{4+\theta}\sigma\|,
\end{split}
\end{equation}
and
\begin{equation}
\begin{split}
\|\p_\xi R\|&\les\Big \| |\xi|^{a-2}\chi \int_0^\xi (\xi-\zeta)\p_\xi^3 \hat{\sigma}(\zeta)d\zeta\Big\|+ \Big \| |\xi|^{a-1}\p_\xi\chi \int_0^\xi (\xi-\zeta)\p_\xi^3 \hat{\sigma}(\zeta)d\zeta\Big\|+\\
&+\||\xi|^{a-1}\chi \int_0^\xi\p_\xi^3 \hat{\sigma}(\zeta)d\zeta\Big\|\\
&\les \lii\||\xi|^{a-2}\xi^2 \chi\|+\lii\||\xi|^{a-1}\p_\xi \chi \xi^2\|+\lii \||\xi|^{a-1}|\xi|\chi\|\\
&\les \|\lanx^{4+\theta}\sigma\|.
\end{split}
\end{equation}
Putting $h(\xi):=|\xi|^{a-1}(1-\chi)$, it follows that $h,\p_\xi h \in L^\infty_\xi$. Hence
\begin{equation}
\begin{split}
E_{6,2}^2 &\les \|\mathcal D_\xi^{\theta} h\|_{L^\infty_\xi}\|\p_\xi \ha\|+\|h\|_{L^\infty_\xi}\|\Dta \p_\xi \ha\|\\
&\les \|x\sigma\|+\||x|^{\theta+1}\sigma\|\\
&\les \|\lanx^{4+\theta}\sigma\|.
\end{split}
\end{equation}

Thus, from the last inequalities

\begin{equation}
\begin{split}\label{E62}
E_{6,2} \les\|\lanx^{4+\theta}\sigma\|.
\end{split}
\end{equation}

The first term on the right-hand side of \eqref{E6plit} can be estimated as

\begin{equation}
\begin{split}\label{E61}
E_{6,1} &\les \|J^2 D^{a-1}(x\sigma)\|+\|D^{a-1}(x\sigma)\|\\
        &\les \|J^{a+1}(x\sigma)\|+\|D^{a-1}(x\sigma)\|\\
        &\les \|J^{2(a+1)}\sigma\|+\|x^2 \sigma\|+\|D^{a-1}(x\sigma)\|,
\end{split}
\end{equation}
where above we used Lemma \eqref{DF2}(with $\mu=\beta=1$). 

The term $\|D^{a-1}(x\sigma)\|$ in \eqref{E61} can be deal by similar way to $E_{6,2}$.

Then, by \eqref{E6plit}, \eqref{E62} and \eqref{E61} we see that

\begin{equation}\label{E6}
E_{6} \les \|J^{2(a+1)}\sigma\|+\|\lanx^{4+\theta}\sigma\|.
\end{equation}

With respect to term $E_8$, from Lemma \ref{DF}

\begin{equation}
\begin{split}\label{sp8}
E_8 \les \|J^{(1+a)\theta}\h D^{a}(x^2\sigma)\|+\||x|^{\theta}\h D^{a}\h (x^2\sigma)\|:=E_{8,1}+E_{8,2}.
\end{split}
\end{equation}

By the Lemma \eqref{DF2}(with $\mu=\theta+\frac{a}{1+a}$) 

\begin{equation}
\begin{split}\label{sp81}
E_{8,1} \leq& \; \|J^{(1+a)\theta+a}(x^2\sigma)\|\\
= & \; \|J^{(1+a)(\theta+\frac{a}{1+a})}(x^2\sigma)\|\\
\les& \; \|J^{(2+\theta+\frac{a}{1+a})(1+a)}\sigma\|+\|\lanx^{2+\theta+\frac{a}{1+a}}\sigma\|\\
\les& \; \|J^{(3+\theta)(1+a)}\sigma\|+\|\lanx^{3+\theta}\sigma\|.
\end{split}
\end{equation}

Also using Lemma \eqref{DF1}(with $\beta=0$ and $r=2+\theta$)

\begin{equation}
\begin{split}\label{sp82}
E_{8,2} \leq& \; \|J^{a(1-\frac{\theta}{2+\theta})^{-1}}(x^2\sigma)\|+\||x|^{2+\theta}x^2 \sigma\|\\
= & \; \|J^{a(1+\frac{\theta}{2})}(x^2\sigma)\|+\|\lanx^{4+\theta}\sigma\|\\
\les& \; \|J^{a(3+\frac{\theta}{2})}\sigma\|+\|\lanx^{4+\theta}\sigma\|+\||x|^{3+\frac{\theta}{2}}\sigma\|\\
\les& \; \|J^{(3+\theta)(1+a)}\sigma\|+\|\lanx^{4+\theta}\sigma\|.
\end{split}
\end{equation}

About the terms $E_k$, $k\in \{1,...,11\}$, $k\neq 1,6,8$, using Lemmas \ref{DF}, \ref{DF1} and \ref{DF2} we conclude that

\begin{equation}
\begin{split}\label{sp82k}
E_{k} \les& \; \|J^{(1+a)(4+\theta)}\sigma\|+\||x|^{4+\theta}\sigma\|.
\end{split}
\end{equation}

Therefore

\begin{equation}
\begin{split}\label{sigma}
\||x|^{4+\theta}U(t)\sigma\|\les \; \|J^{(1+a)(4+\theta)}\sigma\|+\||x|^{4+\theta}\sigma\|.
\end{split}
\end{equation}







Thus, from \eqref{lz}, \eqref{l1z} it's possible to use \eqref{sigma}, with $z$ instead of $\sigma$ we obtain

\begin{equation}
\begin{split}\label{lf}
\||x|^{4+\theta}U(t-\tau)z(\tau)\|\les& \; \|J^{(1+a)(4+\theta)}z(\tau)\|+\||x|^{4+\theta}z(\tau)\|\\
\les& \; \|J^{(1+a)(4+\theta)}(\p_x w(u_1+u_2))\|+\||x|^{4+\theta}\p_x w(u_1+u_2)\|+\\
&+ \|J^{(1+a)(4+\theta)}(w\p_x(u_1+u_2))\|+\||x|^{4+\theta}w\p_x(u_1+u_2)\|\\
&:= \; F_1+\cdots+ F_4.
\end{split}
\end{equation}

Since $H^{(1+a)(4+\theta)}$ is Banach algebra it follows that
\begin{equation}
\begin{split}\label{j1}
F_1&\les \|J^{(1+a)(4+\theta)}\p_x w\|\|J^{(1+a)(4+\theta)}(u_1+u_2)\|\\
&\les \|J^{5+\theta+a(4+\theta)} w\|(\|J^{(1+a)(4+\theta)}u_1\|+\|J^{(1+a)(4+\theta)}u_2\|)\\
&\les M^2.
\end{split}
\end{equation}

As in \eqref{j1} we can find

\begin{equation}
\begin{split}\label{j4}
F_3\les M^2.
\end{split}
\end{equation}

By Sobolev's embedding and Lemma \ref{interx} (with $\alpha=b=2+\theta$)

\begin{equation}
\begin{split}\label{j2}
F_2&\les \ \||x|^{1+\theta}(u_1+u_2)\|_{L^\infty_x}\||x|^3\p_x w\|\\
   &\les \ \|J(\lanx^{1+\theta}(u_1+u_2))\|\|J(\lanx^{3}w)\|\\
&\les \ (\|J^{2+\theta}(u_1+u_2)\|+\|\lanx^{2+\theta}(u_1+u_2)\|)(\|J^{4}w\|+\|\lanx^{4}w\|)\\
&\les \ M^2.
\end{split}
\end{equation}
Also, from Sobolev's embedding and Lemma \ref{interx}
\begin{equation}
\begin{split}\label{j3}
F_4\leq& \ \||x|^\theta \p_x(u_1+u_2)\|_{L^\infty_x}\|x^{4} w\|\\
\les& \ \|J^2(\lanx^\theta (u_1+u_2))\|\|x^4 w\|\\
\les& \  (\|J^{4}(u_1+u_2)\|+\|\lanx^{2\theta}(u_1+u_2)\|)\|x^4 w\|\\
\les& \ M^2.
\end{split}
\end{equation}

From \eqref{lf}--\eqref{j3}

\begin{equation}
\begin{split}\label{j5}
\||x|^{4+\theta}U(t-\tau)z(\tau)\|\les M^2.
\end{split}
\end{equation}

Therefore, gathering together \eqref{intxi}, \eqref{sigma} and \eqref{j5} it follows that for all $t\in [0,T]$

\begin{equation}
\begin{split}
\||x|^{4+\theta}w(t)\|&\les \|J^{(1+a)(4+\theta)}\sigma\|+\||x|^{4+\theta}\sigma\| +\int_0^tM^2 d\tau\\
&\les  \|J^{(1+a)(4+\theta)}\phi\|+\|J^{(1+a)(4+\theta)}\varphi\|+\||x|^{4+\theta}\phi\|+\||x|^{4+\theta}\varphi\|+tM^2,
\end{split}
\end{equation}


which ends the proof of case $r=4$.

In the next, we consider the case $r=5$. First we assume that $\theta>0$. Let $u_1$ and $u_2$ be the solutions of the IVP \eqref{gbo}, with initial data $\phi$ and $\varphi$, respectively. From Theorem A it follows that there exists $T>0$ such that $u_1,u_2\in C([-T,T];Z_{s,5/2+a^{-}})$ and by using Theorem \ref{no} we see that $w=u_1-u_2\in C([-T,T];Z_{s,4+\theta})$.

Hence

\begin{equation}\label{M1}
M_1:=\sup_{[-T,T]}\{\|u_1(t)\|_{Z_{s,5/2}}+\|u_2(t)\|_{Z_{s,5/2}}+\|w(t)\|_{Z_{s,4+\theta}}\}<\infty.
\end{equation}

From \eqref{int} we obtain, by using Plancherel identity 

\begin{equation}\label{intxi5}
\begin{split}
D_\xi^\theta \p_\xi^5(\widehat{w(t)})=&D_\xi^\theta \p_\xi^5(\psi(\xi,t)\ha)-\int_0^t D_\xi^\theta \p_\xi^5(\psi(\xi,t-\tau)\hat{z}(\tau))d\tau.
\end{split}
\end{equation} 
The identity \eqref{F6}, with $\sigma$ instead of $f$, allow us to write

\begin{equation}
\begin{split}\label{dtheta5}
\||x|^{5+\theta}U(t)\sigma\|\les & \ \|\D(\psi |\xi|^{a-3} \ha)\|+\|\D(\psi  |\xi|^{2a-1}\s\ha)\|+\|\D(\psi |\xi|^{3a+1} \ha)\|+\\
&+\|\D(\psi |\xi|^{4a+3}\s \ha)\|+\|\D(\psi |\xi|^{4a+2} \s \ha)\|+\|\D(\psi |\xi|^{5(a+1)} \ha)\|+\\
&+\|\D(\psi |\xi|^{a-2} \s \p_\xi \ha)\|+\|\D(\psi |\xi|^{2a} \p_\xi\ha)\|+\|\D(\psi |\xi|^{3a+2} \s \p_\xi \ha)\|+\\
&+\|\D(\psi |\xi|^{4(a+1)} \p_\xi \ha)\|+\|\D(\psi |\xi|^{a-1} \p_\xi^2 \ha)\|+\|\D(\psi |\xi|^{2a+1} \s \p_\xi^2 \ha)\|+\\
&+\|\D(\psi |\xi|^{3(a+1)}\p_\xi^2 \ha)\|+\|\D(\psi |\xi|^{a} \s \p_\xi^3 \ha)\|+\|\D(\psi |\xi|^{2(a+1)}\p_\xi^3 \ha)\|+\\
&+\|\D(\psi |\xi|^{a+1}\p_\xi^4 \ha)\|+\|\D(\psi \p_\xi^5 \ha)\|\\
:=& \ \tilde B_1+\cdots+\tilde B_{18}.
\end{split}
\end{equation}

In the last inequality, we will estimate only the terms that present more difficulties.

Using Lemma \eqref{DF} we see that

\begin{equation}\label{B1}
\tilde B_1\les \ \|J^{(1+a)\theta}D^{a-3}\sigma\|+\||x|^\theta D^{a-3}\sigma\|:=\tilde B_{1,1}+\tilde B_{1,2}.
\end{equation}

Since the definition of $\chi$ in \eqref{chi}

\begin{equation}
\begin{split}\label{splite25}
\tilde B_{1,2} \les &\|D^{\theta}_\xi( |\xi|^{a-3}\chi \ha)\|+\|\D( |\xi|^{a-3}(1-\chi)\ha)\|:= \tilde B_{1,2}^1+ \tilde B_{1,2}^2.
\end{split}
\end{equation}

In view of identity \eqref{taylor} we can write

\begin{equation}
\begin{split}\label{B121}
\tilde B_{1,2}^1=&\Big\|D^{\theta}_\xi\big(|\xi|^{a-3} \chi \int_0^\xi (\xi-\zeta)\partial_\xi ^2 \hat{\sigma}(\zeta)d\zeta \big)\Big\|\\
\les& \ \Big\|D^{\theta}_\xi \big(|\xi|^{a-3} \chi \int_0^\xi (\xi-\zeta)(\partial_\xi ^2 \hat{\sigma}(\zeta)-\p_\xi^2 \hat{\sigma}(0))d\zeta \big)\Big\|+\Big\|D^{\theta}_\xi(|\xi|^{a-3} \chi \int_0^\xi (\xi-\zeta)\p_\xi^2 \hat{\sigma}(0)d\zeta )\Big\|\\
=&\Big\|D^{\theta}_\xi \big(\underbrace{|\xi|^{a-3} \chi \int_0^\xi (\xi-\zeta)\int_0^\zeta \p_\xi^3 \hat{\sigma}(\eta)d\eta d\zeta}_{S} \big)\Big\|+\frac12\Big\|D^{\theta}_\xi(|\xi|^{a-3} \chi \xi^2 \p_\xi^2 \hat{\sigma}(0))\Big\|\\
\les & \|\D S\|+\|\p_\xi^2 \ha\|_{L^\infty_\xi}\|\D (|\xi|^{a-1} \chi)\|.
\end{split}
\end{equation}

To deal with the right-hand side of the last inequality we observe that by Proposition \ref{DsteinL3}, with $\gamma=a-1/2$ and $\epsilon=a-1/2-\theta$ we obtain $\D (|\xi|^{a-1} \chi)\in L^2$.

We also have, by using Sobolev embedding

\begin{equation}
\begin{split}
\|S\|&\les \Big \| |\xi|^{a-3}\chi \int_0^\xi (\xi-\zeta)\int_0^\zeta|\p_\xi^3 \hat{\sigma}(\eta)|d\eta d\zeta\Big\|\\
&\les \|\p_\xi^3 \ha\|_{L^\infty_\xi}\||\xi|^{a-3}\chi |\xi|^3\|\\
&\les \|\lanx^{5+\theta} \sigma\|\| |\xi|^{a}\chi\|\\
&\les \|\lanx^{5+\theta} \sigma\|,
\end{split}
\end{equation}
and 
\begin{equation}
\begin{split}
\|\p_\xi S\|\les& \ \Big\| |\xi|^{a-4}\chi \int_0^\xi (\xi-\zeta)\int_0^\zeta\p_\xi^3 \hat{\sigma}(\eta)d\eta d\zeta\Big\|+ \Big \| |\xi|^{a-3}\p_\xi\chi \int_0^\xi (\xi-\zeta)\int_0^\zeta\p_\xi^3 \hat{\sigma}(\eta)d\eta d\zeta\Big\|\\
&+\Big \| |\xi|^{a-3}\chi \int_0^\xi \int_0^\zeta\p_\xi^3 \hat{\sigma}(\eta)d\eta d\zeta\Big\|\\\\
&\les \|\p_\xi^3 \ha\|_{L^\infty_\xi}(\||\xi|^{a-4}\chi |\xi|^3\|+\||\xi|^{a-3}\p_\xi \chi |\xi|^3\|+\||\xi|^{a-3}\chi \xi^2\|)\\
&\les \|\lanx^{5+\theta} \sigma\|.
\end{split}
\end{equation}

Hence, from the interpolation inequality $\|\D S\|\les \|S\|^{1-\theta}\|\p_\xi S\|^{\theta}$ and \eqref{B121} it follows that
\begin{equation}
\tilde B_{1,1}^2\les \|\lanx^{5+\theta} \sigma\|.
\end{equation}

In a similar way to the term $E_{1,2}^2$ in \eqref{splite2} we can get
\begin{equation}\label{B122}
\tilde B_{1,2}^2 \les \|J^2 \sigma\|+\|\lanx^2 \sigma\|.
\end{equation}
 
Also, by writing 

\begin{equation*}
\tilde B_{1,1}\les \|J^3 D^{a-3}\sigma\|\les \|J^a \sigma\|+\|D^{a-3}\sigma\|,
\end{equation*} 
we see that to estimate $B_ {1,1}$ it can proceed by analogous way to the term $B_ {1,2}$.
 Thus, by the above inequalities
 \begin{equation}\label{B1est}
 \tilde B_1\les \|\lanx^{5+\theta}\sigma\|.
 \end{equation}
 
The Lemma \ref{DF} implies that

\begin{equation}
\tilde B_7 \les \|J^{(1+a)\theta}D^{a-2}(x\sigma)\|+\||x|^\theta D^{a-2}\h(x\sigma)\|:=\tilde B_{7,1}+\tilde B_{7,2}.
\end{equation}

Thus, using identity \eqref{taylor2} we can write

\begin{equation}
\begin{split}
\tilde B_{7,2}\les& \ \|\D(|\xi|^{a-2}\s \chi \xi \p_\xi^2 \ha(0))\|+\Big\|\D \Big(\underbrace{|\xi|^{a-2}\s \chi \int_0^\xi (\xi-\zeta)\p_\xi^3 \ha(\zeta)d\zeta}_{T}\Big)\Big\|\\
\les& \ \|\p_\xi^2 \ha\|_{L^\infty_\xi}\|\D(|\xi|^{a-1}\chi)\|+\|T\|+\|\p_\xi T\|\\
\les& \ \|\p_\xi^2 \ha\|_{L^\infty_\xi}\|\D(|\xi|^{a-1}\chi)\|+\|\p_\xi^3 \ha\|_{L^\infty_\xi}\big(\||\xi|^{a-2}\chi \xi^2\|+\||\xi|^{a-3}\chi \xi^2\|+\||\xi|^{a-2}\xi^2 \p_\xi \chi\|+\\
&+\||\xi|^{a-2}\xi \chi\|\big)\\
\les& \|\lanx^{5+\theta}\sigma\|+\|\lanx^{5+\theta}\sigma\|\big(\||\xi|^{a-1}\chi\|+\||\xi|^{a}\chi\|+\||\xi|^a \p_\xi \chi\|\big)\\
\les& \|\lanx^{5+\theta}\sigma\|,
\end{split}
\end{equation}
where the first term on the right-hand side of the inequality above already was treated in \eqref{B121}. About the second term, it was used the interpolation inequality $\|\D T\|\les \|T\|^{1-\theta}\|\p_\xi T\|^{\theta}$ and Young's inequality.

The estimate for the term $\tilde B_{7,1}$ follows in a similar way to the term $\tilde B_{7,2}$. Hence, we can get

\begin{equation}
\tilde B_7 \les \ \|\lanx^{5+\theta}\sigma\|.
\end{equation}

With respect to term $\tilde B_8$, from analogous to above terms we can write

\begin{equation}
\tilde B_8 \les \|J^{(1+a)\theta}D^{a-1}(x^2 \sigma)\|+\||x|^\theta D^{a-1}(x^2 \sigma)\|:=\tilde B_{8,1}+\tilde B_{8,2}.
\end{equation}

Then,

\begin{equation}
\begin{split}\label{}
\tilde B_{8,2} \les \ &\|D^{\theta}_\xi( |\xi|^{a-1}\chi \p_\xi^2\ha)\|+\|\D( |\xi|^{a-1}(1-\chi)\p_\xi^2\ha)\|:=\tilde B_{8,2}^1+\tilde B_{8,2}^2.
\end{split}
\end{equation}

Since 

\begin{equation}
\p_\xi^2 \ha(\xi)-\p_\xi^2 \ha(0)=\int_0^\xi \p_\xi^3 \ha(\zeta)d\zeta,
\end{equation}

it follows that

\begin{equation}
\begin{split}
\tilde B_{8,2}^1 & \les \ \Big\|\D \Big(\underbrace{|\xi|^{a-1}\chi \int_0^\xi \p_\xi^3 \ha(\zeta)d\zeta}_{V}\Big)\Big\|+\|\D(|\xi|^{a-1}\chi \p_\xi^2\ha(0))\|\\
& \les \ \|V\|+\|\p_\xi V\|+\|\p_\xi^2 \ha\|_{L^\infty_\xi}\|\D(|\xi|^{a-1}\chi)\|\\
& \les \|\p_\xi^3 \ha\|_{L^\infty_\xi}(\||\xi|^{a-1}\chi\|+\||\xi|^{a-2}\chi|\xi|\|+\||\xi|^{a-1}\p_\xi\chi |\xi|\|)+\|\p_\xi^2 \ha\|_{L^\infty_\xi}\|\D(|\xi|^{a-1}\chi)\|\\
& \les \|\lanx^{5+\theta}\sigma\|.
\end{split}
\end{equation}

In analogous way to term $E_{6,2}^2$ in \eqref{estE622} we see that

\begin{equation}
\tilde B_{8,2}^2 \les \ \|\lanx^{5+\theta}\sigma\|.
\end{equation}

Also, proceeding in a similar way to the term $\tilde B_{8,2}$ we conclude that

\begin{equation}
\tilde B_{8,1} \les \ \|\lanx^{5+\theta}\sigma\|.
\end{equation}

Hence, by the inequalities above

\begin{equation}
\tilde B_8 \les \|\lanx^{5+\theta}\sigma\|.
\end{equation}

The rest of the other terms in \eqref{dtheta5} can be treated in a similar way to the above. Thus

\begin{equation}
\|\tilde B_k\|\les \|J^{(1+a)\theta}\sigma\|+\|\lanx^{5+\theta}\sigma\|, \quad \mbox{for all} \quad k\neq 1,7,8.
\end{equation}

Then, the above estimates imply that

\begin{equation}
\|\lanx^{5+\theta}U(t)\sigma\|\les \|J^{(1+a)\theta}\sigma\|+\|\lanx^{5+\theta}\sigma\|.
\end{equation} 

With respect to the integral part in \eqref{intxi5}, by \eqref{lz}, \eqref{l1z} and the last inequality we obtain

\begin{equation}
\begin{split}\label{lf5}
\||x|^{5+\theta}U(t-\tau)z(\tau)\|\les& \ \|J^{(1+a)(5+\theta)}z(\tau)\|+\||x|^{5+\theta}z(\tau)\|\\
\les& \ \|J^{(1+a)(5+\theta)}(\p_x w(u_1+u_2))\|+\||x|^{5+\theta}\p_x w(u_1+u_2)\|+\\
&+ \|J^{(1+a)(5+\theta)}(w\p_x(u_1+u_2))\|+\||x|^{5+\theta}w\p_x(u_1+u_2)\|\\
&:=H_1+\cdots+ H_4.
\end{split}
\end{equation}

Proceeding as in \eqref{j1} and \eqref{j4} we see that

\begin{equation}
H_1 \les M_1^2 \quad \mbox{and} \quad H_3 \les M_1^2.
\end{equation}

The Sobolev's embedding and Lemma \ref{interx} (with $\alpha=b=4+\theta$) imply that

\begin{equation}
\begin{split}\label{j25}
H_2&\les  \|x^{2}(u_1+u_2)\|_{L^\infty_x}\||x|^{3+\theta}\p_x w\|\\
   &\les  \|J(\lanx^{2}(u_1+u_2))\|\|J(\lanx^{3+\theta}w)\|\\
&\les  (\|J^{5}(u_1+u_2)\|+\|\lanx^{5/2}(u_1+u_2)\|)(\|J^{4+\theta}w\|+\|\lanx^{4+\theta}w\|)\\
&\les  M^2_1.
\end{split}
\end{equation}

Also, by similar way to above

\begin{equation}
\begin{split}
H_4 \les& \ \||x|^{4+\theta}w\|\|x\p_x(u_1+u_2)\|\\
\les& \ \||x|^{4+\theta}w\|\|J(\lanx(u_1+u_2))\|\\
\les& \ \||x|^{4+\theta}w\| (\|J^2(u_1+u_2)\|+\|\lanx^2(u_1+u_2)\|)\\
\les& \ M_1^2.
\end{split}
\end{equation}

From now, proceeding as at the end of the proof of Theorem \ref{no} we obtain the desired result.

The case $\theta=0$ follows by similar arguments as the above.

This finishes the proof of Theorem \ref{no}.

\end{proof}

\section{Proof of Theorems \ref{2t} and \ref{nolow}}  \label{nolow2t}

In arguments of the next proof, we make use of the conservation quantity \eqref{norm}, beyond the several auxiliary results established in the Section \ref{notation}. 

\begin{proof}[Proof of Theorem \ref{2t}]

Since $\phi\in \dot{Z}_{s,(\frac92+a)-}$ it follows that $u\in C([-T,T]; Z_{s,(\frac72+a)^-})$.

First we will deal with the case $r=4$.

Let $\theta\in (0,1)$, then from our hypothesis $\theta<1/2+a$, the constant

\begin{equation}
N_1:=\sup_{[-T,T]} \|u(t)\|_{Z_{s,3+\theta}}
\end{equation}  

is finite.

Here, we looking at \eqref{F5} with $\phi$ instead of $f$. Next, will deal with the corresponding terms $A_1(t,\xi,\hu)$ and $A_6(t,\xi,\hu)$. 

To deal with $A_1$ we can write

\begin{equation}\label{dA11}
|\xi|^{a-2}\s \hu \psi=|\xi|^{a-2}\s \hu (1-\chi)\psi+|\xi|^{a-2}\s \hu\chi\psi:=A_{1,1}+A_{1,2},
\end{equation}
where $\psi$ is given in \eqref{psidef}.

The second term on the right-hand side of the last equality can be decomposed as

\begin{equation}\label{dA121}
A_{1,2}=|\xi|^{a-2}\s \hu \chi(\psi-1)+|\xi|^{a-2}\s \hu\chi:=A_{1,2}^1+A_{1,2}^2.
\end{equation}

Since $\hu(0)=0$, it follows that

\begin{equation}\label{taylor2t}
\hu(\xi)=\xi\p_\xi \hu(0)+\int_0^\xi (\xi-\zeta)\partial_\xi ^2 \hu(\zeta)d\zeta.
\end{equation}

Hence, from $\partial_\xi ^2 \hu(\zeta)-\p_\xi^2\hu(0)=\int_0^\zeta \p_\xi^3 \hu(\eta)d\eta$, we obtain

\begin{equation}
\begin{split}\label{dA122}
A_{1,2}^2=& \ |\xi|^{a-1}\p_\xi \hu(0)\chi+|\xi|^{a-2}\s \chi\int_0^\xi (\xi-\zeta)\partial_\xi ^2 \hu(\zeta)d\zeta\\
=& \ |\xi|^{a-1}\p_\xi \hu(0)\chi+|\xi|^{a-2}\s\chi\int_0^\xi (\xi-\zeta)\int_0^\zeta \partial_\xi ^3 \hu(\eta)d\eta d\zeta+\\
&+|\xi|^{a-2}\s \chi \partial_\xi ^2 \hu(0)\int_0^\xi (\xi-\zeta)d\zeta\\
=& \ |\xi|^{a-1}\p_\xi \hu(0)\chi+L+\frac12\p_\xi^2\hu(0)|\xi|^a \s\chi,
\end{split}
\end{equation}

where the term $L$ is given by \eqref{E121est}.

To deal with $A_6$ we see that

\begin{equation}\label{dA611}
|\xi|^{a-1} \p_\xi \hu \psi=|\xi|^{a-1}\p_\xi \hu (1-\chi)\psi+|\xi|^{a-1} \p_\xi\hu\chi\psi:=A_{6,1}+A_{6,2}.
\end{equation}
Thus,
\begin{equation}\label{dA621}
A_{6,2}=|\xi|^{a-1}\p_\xi\hu \chi(\psi-1)+|\xi|^{a-1}\p_\xi\hu\chi:=A_{6,2}^1+A_{6,2}^2.
\end{equation}

From $\p_\xi \hu(\xi)-\p_\xi \hu(0)=\int_0^\xi \p_\xi^2 \hu(\zeta)d\zeta$, it follows that

\begin{equation}
\begin{split}\label{dA622}
A_{6,2}^2=& \ |\xi|^{a-1}(\p_\xi \hu(\xi)-\p_\xi \hu(0))\chi+|\xi|^{a-1}\p_\xi\hu(0)\chi\\
        =& \ |\xi|^{a-1}\chi\int_0^\xi\p_\xi^2 \hu(\zeta)d\zeta +|\xi|^{a-1}\p_\xi\hu(0)\chi\\
        =& \ \underbrace{|\xi|^{a-1}\chi \int_0^\xi \int_0^\zeta \p_\xi^3 \hu(\eta)d\eta}_{\tilde A}+\p_\xi^2 \hu(0)|\xi|^a \s \chi +|\xi|^{a-1}\p_\xi\hu(0)\chi.
\end{split}
\end{equation}

Hence, from the above equalities and \eqref{F5} we can write

\begin{equation}
\begin{split}\label{A1A6}
A_1+A_6=& \ tc_a |\xi|^{a-1}\p_\xi\hu(0)\chi+t\Big[c_1 L+ \Big(\frac{c_1}{2}+c_6\Big)\p_\xi^2 \hu(0)|\xi|^a \s \chi\\
&+ c_6 \tilde A+c_1 A_{1,1}+c_1 A_{1,2}^1+c_6 A_{6,1}+c_6 A_{6,2}^1\Big]\\
:=& \ tc_a|\xi|^{a-1}\p_\xi\hu(0)\chi+P,
\end{split}
\end{equation}

where $c_1=-2(2+a)(a^2-1)ai$, $c_6=-(2+a)(1+a)ai$ and $c_a=c_1+c_6$.

Proceeding in a similar way to the proof of Theorem \ref{no} we conclude that $P=P(t,\xi,\hu)$ satisfies $\D P \in L^2$.

The identity \eqref{norm} implies that

\begin{equation}\label{pxikappa}
\p_\xi \widehat \vartheta (\tau,0)=\frac{i}{2}\widehat{u^2}(\tau,0)=\frac{i}{2}\|\phi\|^2,
\end{equation}
where $\vartheta$ is given in \eqref{inte}.

Hence, using \eqref{F5}, \eqref{A1A6} together with the integral equation \eqref{inte}

\begin{equation}\label{d4u}
\begin{split}
\p_\xi^4\widehat{u}(t)=& \ c_a |\xi|^{a-1}\chi \Bigg(\underbrace{-t\p_\xi\hu(0)+\frac{i}{2}\int_0^t (t-\tau)\|\phi\|^2d\tau}_{F(t)}\Bigg)+ P(t,\xi,\hu)-\int_0^t  P(t-\tau,\xi,\widehat{\vartheta}(\tau))d\tau\\
&+ \sum_{\overset{1<j\leq 11}{j\not=6}} A_j(t,\xi, \hu)-\sum_{\overset{1<j \leq 11}{j\not=6}}\int_0^t A_j(t-\tau,\xi, \widehat{\vartheta}(\tau)).
\end{split}
\end{equation}

A simple computation gives us 

\begin{equation*}
\begin{split}
F(t)= \ it\Big(\int_{-\infty}^{\infty} x \phi(x)dx+\frac{\|\phi\|^2}{4}t\Big).
\end{split}
\end{equation*}
Thus, if 

$$t=t^{*}=\frac{-4}{\|\phi\|^2}\int_{-\infty}^{\infty}x\phi(x)dx,$$

then 
\begin{equation}\label{Func}
F(t^{\ast})=0.
\end{equation}

As in the proof of Theorem \ref{no} we conclude that

\begin{equation}
\|\D P(t,\xi,\hu)\| \les \|J^{(4+\theta)(1+a)}\phi\|+\|\lanx^{4+\theta}\phi\|
\end{equation}
and
\begin{equation}
\|\D A_j(t,\xi,\hu)\| \les \|J^{(4+\theta)(1+a)}\phi\|+\|\lanx^{4+\theta}\phi\|, \quad \mbox{for} \quad 1<j\leq 11, \ j\not=6.
\end{equation}

Also, from Sobolev's embedding and Lemma \ref{interx} (with $\alpha=b=3+\theta$)

\begin{equation}
\begin{split}\label{DPu}
\|\D P(t-\tau,\xi,\widehat{\vartheta}(\tau))\|& \les \ \|J^{(4+\theta)(1+a)}(uu_x)\|+\|\lanx^{4+\theta}uu_x\|\\
&\les \|J^{(4+\theta)(1+a)+1}u^2(\tau)\|+\|\lanx^2 u\|_{L^\infty}\|\lanx^{2+\theta}u_x\|\\
&\les \|J^{(4+\theta)(1+a)+1}u(\tau)\|^2+\|J(\lanx^{2+\theta}u)\|\big(\|J(\lanx^{2+\theta}u)\|+\|\lanx^{1+\theta}u\|\big)\\
&\les \|J^{(4+\theta)(1+a)+1}u(\tau)\|^2 +\|J^{3+\theta}u\|^2+\|\lanx^{3+\theta}u\|^2\\
&\les N_1^2,
\end{split}
\end{equation}

and, as above

\begin{equation}\label{DAju}
\|\D A_j(t-\tau,\xi, \widehat{\vartheta}(\tau))\|\les N_1^2, \quad \mbox{for} \quad 1<j\leq 11, \ j\not=6.
\end{equation}

By \eqref{DPu} and \eqref{DAju}

\begin{equation}\label{Dintu}
\Big\|\int_0^t  \D P(t-\tau,\xi,\widehat{\vartheta}(\tau))d\tau\Big\|\les |t|N_1^2 \quad \mbox{and} \quad \Big\|\int_0^t  \D A_j(t-\tau,\xi, \widehat{\vartheta}(\tau))\Big\|\les |t|N_1^2,
\end{equation}

where $1<j\leq 11$, $j\not=6$.

Therefore, taking $\D$ in both sides of \eqref{d4u} and using \eqref{Func}--\eqref{Dintu} we obtain

\begin{equation}
u(t^{\ast})\in Z_{s,4+\theta}.
\end{equation}

This ends the case $r=4$.

Next, we will let's attention to the case $r=5$. 

In view of $5+\theta<9/2+a$, we see that $\theta<a-1/2$. Note that here, we must have $1/2<a<1$. We also have that the constant 

\begin{equation}
N_2:=\sup_{[-T,T]} \|u(t)\|_{Z_{s,4+\theta}}
\end{equation}  

is finite.

Now, we will employ \eqref{F6} with $\phi$ instead of $f$. We will deal with the terms $B_1(t,\xi,\hu)$ and $B_7(t,\xi,\hu)$.

To estimate $B_1$ we can write

\begin{equation}\label{dB11}
|\xi|^{a-3} \hu \psi=|\xi|^{a-3} \hu (1-\chi)\psi+|\xi|^{a-3} \hu\chi\psi:=B_{1,1}+B_{1,2}.
\end{equation}

Thus,

\begin{equation}\label{dB121}
B_{1,2}=|\xi|^{a-3}\hu \chi(\psi-1)+|\xi|^{a-3}\hu\chi:=B_{1,2}^1+B_{1,2}^2.
\end{equation}

Using \eqref{taylor2t} we obtain

\begin{equation}
\begin{split}\label{dA122}
B_{1,2}^2=& \ |\xi|^{a-2}\s\p_\xi \hu(0)\chi+|\xi|^{a-3} \chi\int_0^\xi (\xi-\zeta)\partial_\xi ^2 \hu(\zeta)d\zeta\\
=& \  |\xi|^{a-2}\s\p_\xi \hu(0)\chi+|\xi|^{a-3}\chi\int_0^\xi (\xi-\zeta)\int_0^\zeta \partial_\xi ^3 \hu(\eta)d\eta d\zeta+|\xi|^{a-3} \chi \partial_\xi ^2 \hu(0)\int_0^\xi (\xi-\zeta)d\zeta\\
=& \ |\xi|^{a-2}\s\p_\xi \hu(0)\chi+S+\frac12\p_\xi^2\hu(0)|\xi|^{a-1} \chi,
\end{split}
\end{equation}

where $S$ is given by \eqref{B121}.

For $B_7$ we can write

\begin{equation}\label{dB711}
|\xi|^{a-2} \s \p_\xi \hu \psi=|\xi|^{a-2}\s \p_\xi \hu (1-\chi)\psi+|\xi|^{a-2}\s \p_\xi\hu\chi\psi:= B_{7,1}+B_{7,2}.
\end{equation}

Then,
\begin{equation}\label{dB721}
B_{7,2}=|\xi|^{a-2}\s\p_\xi\hu \chi(\psi-1)+|\xi|^{a-2}\s \p_\xi\hu\chi:=B_{7,2}^1+B_{7,2}^2.
\end{equation}

Since $\p_\xi \hu(\xi)-\p_\xi \hu(0)=\int_0^\xi \p_\xi^2 \hu(\zeta)d\zeta$, we see that

\begin{equation}
\begin{split}\label{dB722}
B_{7,2}^2=& \ |\xi|^{a-2}\s(\p_\xi \hu(\xi)-\p_\xi \hu(0))\chi+|\xi|^{a-2}\s\p_\xi\hu(0)\chi\\
        =& \ |\xi|^{a-2}\s\chi\int_0^\xi\p_\xi^2 \hu(\zeta)d\zeta +|\xi|^{a-2}\s\p_\xi\hu(0)\chi\\
        =& \ \underbrace{|\xi|^{a-2}\s\chi \int_0^\xi \int_0^\zeta \p_\xi^3 \hu(\eta)d\eta d\zeta}_{\tilde B}+\p_\xi^2 \hu(0)|\xi|^{a-1}\chi +|\xi|^{a-2}\s\p_\xi\hu(0)\chi.
\end{split}
\end{equation}

Then, by the above equalities and \eqref{F6} we have

\begin{equation}
\begin{split}\label{B1B7}
B_1+B_7=& \ td_a |\xi|^{a-2}\s\p_\xi\hu(0)\chi+t\Big[d_1 S+ \Big(\frac{d_1}{2}+d_7\Big)\p_\xi \hu(0)|\xi|^{a-1}\chi\\
&+ d_7 \tilde B+d_1 B_{1,1}+d_1 B_{1,2}^1+d_7 B_{7,1}+d_7 B_{7,2}^1\Big]\\
:=& \ td_a|\xi|^{a-2}\s\p_\xi\hu(0)\chi+Q,
\end{split}
\end{equation}

where $d_1=-(a^2-4)(a^2-1)ai$, $d_7=-10(2+a)(a^2-1)ai$ and $d_a=d_1+d_7$.

Thus, as in the proof of Theorem \ref{no} follows that $\D Q \in L^2$.

By \eqref{F6}, \eqref{B1B7} and \eqref{pxikappa}

\begin{equation}\label{d5u}
\begin{split}
\p_\xi^5\widehat{u}(t)=& \ d_a |\xi|^{a-2}\s\chi \Bigg(\underbrace{-t\p_\xi\hu(0)+\frac{i}{2}\int_0^t (t-\tau)\|\phi\|^2d\tau}_{F(t)}\Bigg)+ Q(t,\xi,\hu)\\
&-\int_0^t  Q(t-\tau,\xi,\widehat{\vartheta}(\tau))d\tau+ \sum_{\overset{1<j\leq 18}{j\not=7}} B_j(t,\xi, \hu)-\sum_{\overset{1<j \leq 18}{j\not=7}}\int_0^t B_j(t-\tau,\xi, \widehat{\vartheta}(\tau)).
\end{split}
\end{equation}

By the analogous way, to the proof of Theorem \ref{no} we see that

\begin{equation}
\|\D Q(t,\xi,\hu)\| \les \|J^{(5+\theta)(1+a)}\phi\|+\|\lanx^{5+\theta}\phi\|
\end{equation}
and
\begin{equation}
\|\D B_j(t,\xi,\hu)\| \les \|J^{(5+\theta)(1+a)}\phi\|+\|\lanx^{5+\theta}\phi\|, \quad \mbox{for} \quad 1<j\leq 18, \ j\not=7.
\end{equation}

Using Sobolev's embedding and Lemma \ref{interx} (with $\alpha=b=4+\theta$)

\begin{equation}
\begin{split}\label{DQu}
\|\D Q(t-\tau,\xi,\widehat{\vartheta}(\tau))\|& \les \ \|J^{(5+\theta)(1+a)}(uu_x)\|+\|\lanx^{5+\theta}uu_x\|\\
&\les \|J^{(5+\theta)(1+a)+1}u^2(\tau)\|+\|\lanx^2 u\|_{L^\infty}\|\lanx^{3+\theta}u_x\|\\
&\les \|J^{(5+\theta)(1+a)+1}u(\tau)\|^2+\|J(\lanx^{3+\theta}u)\|\big(\|J(\lanx^{3+\theta}u)\|+\|\lanx^{2+\theta}u\|\big)\\
&\les \|J^{(5+\theta)(1+a)+1}u(\tau)\|^2 +\|J^{4+\theta}u\|^2+\|\lanx^{4+\theta}u\|^2\\
&\les N_2^2,
\end{split}
\end{equation}

and

\begin{equation}\label{DBju}
\|\D B_j(t-\tau,\xi, \widehat{\vartheta}(\tau))\|\les N_2^2, \quad \mbox{for} \quad 1<j\leq 18, \ j\not=7.
\end{equation}

The rest of the proof follows as in the case $r=4$.

This concludes the proof.

\end{proof}


Now, we deal with the proof of Theorem \ref{nolow}. With objective to improve the Theorem A, we also obtained the continuous dependence upon initial data in $Z_{s,r}$. 

\begin{proof}[Proof of Theorem \ref{nolow}]



From \eqref{F2}--\eqref{F5}, with $\phi$ instead of $f$, by using similar ideas used to obtain Theorem \ref{no} it follows that for all
$0<r<5/2+a$

\begin{equation}\label{52}
\|\lanx^r U(t)\phi\|\les \|J^{r(a+1)}\phi\|+\|\lanx^r \phi\|, \quad \mbox{if} \quad \phi \in Z_{s,r},
\end{equation}

and for any $5/2+a\leq r<7/2+a $

\begin{equation}\label{72}
\|\lanx^r U(t)\phi\|\les \|J^{r(a+1)}\phi\|+\|\lanx^r \phi\|, \quad \mbox{if} \quad \phi\in \dot{Z}_{s,r}.
\end{equation}

Let $u$ be the solution of the IVP \eqref{gbo} with initial data $\phi$.

First, we will deal with the case $0<r<5/2+a$. Here, we can put $r:=r_k$, where 

\[
r_k= 
  \begin{cases}
      \frac{k\theta+1}{2}, & \mbox{if} \ \  k=1,...,4 \\
      5/2+\frac{a\theta}{2}, & \mbox{if} \ \ k=5\\
      5/2+a\theta, & \mbox{if} \ \ k=6, 
  \end{cases}
\]

for some $\theta\in (0,1)$.

To obtain our result we will use induction in $k$. For $k=1$ the Theorem A (i)) implies that $\lanx^{r_1}u \in C([-T,T];L^2)$. Also assume that $\lanx^{r_{k-1}}u \in C([-T,T];L^2)$, for some $k=2,...,6$. Hence, the constant

\begin{equation}
M_k=\sup_{[-T,T]}\big\{\|J^{r_k (1+a)+1}u\|^2+\|\lanx^{r_{k-1}}u\|+\|J^2 u\|\big\}<\infty.
\end{equation}

We recall that the integral equation associated to the IVP \eqref{gbo} is given by

\begin{equation}\label{In}
u(t)=U(t)\phi-\frac12\int_0^t U(t-\tau)\p_x u^2(\tau) d\tau, \quad  t\in [-T,T].
\end{equation}


Inequality \eqref{52} yields us

\begin{equation}
\begin{split}\label{Uphi}
\|\lanx^{r_k}U(t)\phi\|&\les \|J^{r_{k}(1+a)}\phi\|+\|\lanx^{r_k} \phi\|.\\
\end{split}
\end{equation}

On the other hand,  from \eqref{Uphi}, with $\p_x u^2(\tau)$ instead of $\phi$, assuming that $2\leq k\leq 4$ and using Holder's inequality follows that

\begin{equation}
\begin{split}\label{U}
\|\lanx^{r_k}U(t-\tau)\p_x u^2(\tau)\|&\les \|J^{r_k (1+a)}\p_x u^2(\tau)\|+\|\lanx^{r_k} \p_x u^2(\tau)\|\\
                       &\les \|J^{r_k (1+a)+1}u(\tau)\|^2+\|\lanx^{\theta/2}\p_x u (\tau)\|_{L^\infty_x}\|\lanx^{\frac{(k-1)\theta+1}{2}} u (\tau)\|\\
                       &\les M_k^2+(\|J(\lanx^{\theta/2}u(\tau))\|+\|u(\tau)\|)\|\lanx^{\frac{(k-1)\theta+1}{2}} u (\tau)\|\\
                       &\les M_k^2+(\|J^2 u(\tau)\|+\|\lanx^{\theta}u(\tau)\|)\|\lanx^{r_{k-1}} u (\tau)\|\\
                       &\les M_k^2,
\end{split}
\end{equation}
where above we also used Sobolev embedding and Lemma \ref{interx}. 

Then, by \eqref{In}--\eqref{U}

\begin{equation}\label{rk}
\|\lanx^{r_k}u(t)\|\les \|J^{r_{k}(1+a)}\phi\|+\|\lanx^{r_k} \phi\|+|t|M_k^2, \quad t\in [-T,T].
\end{equation}

From \eqref{rk} and proceeding as in \cite{AP} we can show that $\lanx^{r_k}u \in C([-T,T];L^2)$.

By analogous way it's possible to see that if $\lanx^{r_{k-1}}u \in C([-T,T];L^2)$, then $\lanx^{r_k}u \in C([-T,T];L^2)$, for $5\leq k \leq 6$. 
Therefore, we conclude that $\lanx^{r}u\in C([-T,T];L^2)$, where $0<r<5/2+a$.

The case $5/2+a\leq r <7/2+a$ is similar to the above.

This ends the proof of 1) and 2).

Next, we will deal with the part 3). With the goal to obtain continuous dependence, letting $u,v\in C([-T,T];Z_{s,r})$ solutions of the IVP \eqref{gbo} with initial data $\phi$ and $\varphi$, respectively. 
Then, from the persistence property we deduce $\lanx^{r}(u-v)\in C([-T,T];L^2)$. Thus the constant

\begin{equation}
\tilde M :=\sup_{[-T,T]}\{\|u(t)\|_{Z_{s,r}}+\|v(t)\|_{Z_{s,r}}\}<\infty.
\end{equation}

Our hypothesis implies that $s>3/2$, hence using \eqref{52}, Holder's inequality and Sobolev embedding we get

\begin{equation}
\begin{split}
\|\lanx^r (u-v)\|\leq& \ \|\lanx^r U(t)(\phi-\varphi)\|+\int_0^t \|\lanx^r U(t-\tau)\p_x (u^2-v^2)\|d\tau\\
                 \les& \ \|J^{r(1+a)}(\phi-\varphi)\|+\|\lanx^r (\phi-\varphi)\|+\int_0^t\big(\|J^{s}(u^2-v^2)\|+\|\lanx^r \p_x(u^2-v^2)\|\big)d\tau\\
                 \les& \ \|\phi-\varphi\|_{Z_{s,r}}+T\sup_{[-T,T]}\|u+v\|_{H^s}\sup_{[-T,T]}\|u-v\|_{H^s}+\\
                 &+\int_0^t \big(\|\p_x (u+v)\|_{L^\infty}\|\lanx^r (u-v)\|+\|\lanx^r (u+v)\|\|\p_x(u-v)\|_{L^\infty_x}\big)d\tau\\
                 \les& \ \|\phi-\varphi\|_{Z_{s,r}}+T\tilde M \sup_{[-T,T]}\|u-v\|_{H^s}+\tilde M \int_0^t \|\lanx^r (u-v)\|d\tau.
\end{split}
\end{equation}
Thus, the desired result follows from Gronwall's Lemma and the continuous dependence of the solution in $C([-T,T];H^s)$ with initial data in $H^s$. This shows that the IVP \eqref{gbo} is LWP in $Z_{s,r}$. If $a>1/3$, the global well-posedness upon initial data in $H^s$ can be found in \cite{Guogbo}. Hence, as in the above we conclude the global well-posedness in $Z_{s,r}$.

The proof of 4) follows by analogous way.

This ends the proof.

\end{proof}


\section{Proof of Theorems \ref{low5} and \ref{low5theta}}\label{NBO}

In the following we will stablish our decay results for the Benjamin-Ono equation. Here we use the truncated weights $\lanxN$ defined in the Section \ref{notation}.   

\begin{proof}[Proof of Theorem \ref{low5}]

Let $u_1$ and $u_2$ solutions of the IVP \eqref{gbo} with initial data $\phi$ and $\varphi$, respectively. Here, as in the previous proofs, we set $\sigma=\phi-\varphi$ and  $w=u_1-u_2$. Thus, $w$ satisfies the following linear equation

\begin{equation}\label{bw}
w_{t}+\mathcal{H}\partial_{x}^{2}w+u_1\p_x w+w\p_x u_2 =0, \;\;x,t\in \R.
\end{equation}

The global well-posedness gives us $u_1,u_2\in C([-T,T];H^s)$, for all $T>0$. In addition, from \cite[Theorem 1]{GermanPonce} we see that $w\in L^\infty([-T,T];Z_{s,4})$. Hence, the constant

\begin{equation}
P= \sup_{[-T,T]}\{\|u_1(t)\|_{H^s}+\|u_2(t)\|_{H^s}+\|w(t)\|_{Z_{s,4}}\}<\infty.
\end{equation}

Identities \eqref{I1} and \eqref{I2} imply that the solution $w$ of \eqref{bw} satisfies the following conservation laws

\begin{equation}\label{law1}
\int w(x,t)dx=\int \sigma(x)dx
\end{equation}
and
\begin{equation}\label{law2}
\frac{d}{dt}\int x w(x,t)dx=\frac12(\|\phi\|^2-\|\varphi\|^2).
\end{equation}

Then, from \eqref{law1}, \eqref{law2} and the hypothesis \eqref{norma}--\eqref{v1} it follows that 

\begin{equation}\label{consw}
\int w(x,t)dx=0
\end{equation}

and

\begin{equation}\label{consxw}
\int x w(x,t)dx=0,
\end{equation}

for all $t$ in which the solution there exists.

Multiplying \eqref{bw} by $\lanxN^{2}x^8 w$ and integrating on $\R$ we obtain

\begin{equation}\label{intwn}
\begin{split}
\frac12 \|\lanxN x^4 w\|^2+\int \lanxN x^4\h \p_x^2 w\lanxN x^4 w+\int \lanxN^{2}&x^8(u_1\p_x w+w\p_x u_2)=0.
\end{split}
\end{equation}

Following, our goal is to estimate the second term on the identity above.

The identities

\begin{equation*}
x\h \p_x^2 w=\h \p_x^2(xw)-2\h \p_x w,
\end{equation*} 

\begin{equation*}
x^2\h \p_x^2 w=\h \p_x^2(x^2w)-4\h \p_x (xw)+2\h w,
\end{equation*} 

\begin{equation*}
x^3 \h \p_x^2 w=\h\p_x^2(x^3 w)-6\h \p_x (x^2 w)-2\h (xw)
\end{equation*}

together with \eqref{consxw} gives us

\begin{equation*}
x^4 \h \p_x^2 w=\h\p_x^2(x^4 w)-8\h \p_x (x^3 w)+4\h (x^2w).
\end{equation*}

Thus, 

\begin{equation}
\begin{split}
\lanxN x^4 \h \p_x^2 w&=\lanxN \h\p_x^2(x^4 w)-8\lanxN \h \p_x (x^3 w)+4\lanxN \h (x^2w)\\
&:= \ \mathcal{A}+\mathcal B+\mathcal C.
\end{split}
\end{equation}

To estimate the first term in the last identity we can write

\begin{equation}\label{ma}
\begin{split}
\mathcal A&=[\lanxN;\h]\p_x^2(x^4 w)+\h(\lanxN\p_x^2(x^4 w))\\
&=\  \mathcal A_1 +\h\p_x^2 (\lanxN x^4 w)-2\h (\p_x \lanxN \p_x(x^4 w))-\h(\p_x^2\lanxN x^4 w)\\
&= \ \mathcal A_1+\cdots+\mathcal A_4.
\end{split}
\end{equation}

Thus, by the Calder\'on commutator estimate (see Theorem 6 in \cite{FonsecaPonce} and references therein) we get

\begin{equation}\label{ma1}
\|\mathcal A_1\|\les \|\p_x^2\lanxN \|_{L^\infty}\|x^4 w\|\les P.
\end{equation}

Also, 

\begin{equation}\label{ma4}
\|\mathcal A_4\|\les \|\p_x^2\lanxN x^4 w\|\les \|\p_x^2\lanxN \|_{L^\infty}\|x^4 w\|\les P.
\end{equation}

By returning the term $\mathcal A_2$ in \eqref{intwn} we get

\begin{equation}\label{ma2}
\int \h \p_x^2 (\lanxN x^4w)\lanxN x^4 w=0.
\end{equation}
To estimate $\mathcal A_3$, the inequality $|x\p_x \lanxN|\les \lanxN $ yields us 

\begin{equation}\label{ma3}
\begin{split}
\|\mathcal A_3\| &\les \ \| \p_x \lanxN x^3 w\|+\| \p_x \lanxN x^4 \p_x w\|\\
        &\les \ \|\lanxN x^2 w\|+\underbrace{\|\lanxN x^3 \p_x w\|}_{F}\\
        &\les P+\|J(\lanx^3 \lanxN w)\|\\
                &\les P+\|J^4(\lanxN w)\|+\|\lanx^4 \lanxN w\|\\
                &\les P+\|J^5 w\|+\|\lanxN^5 w\|+\|\lanx^4 \lanxN w\|\\
                &\les P+\|\lanx^4 \lanxN w\|,
\end{split}
\end{equation}
where above we used Lemma \ref{interx} (with $\alpha=b=4$ and $\alpha=b=5$).

To estimate term $\mathcal B$ we can write

\begin{equation}
\begin{split}
\lanxN \h \p_x(x^3w)&=\ [\lanxN; \h]\p_x (x^3 w)+\h (\lanxN \p_x(x^3 w))\\
                   &:= \mathcal B_1+\mathcal B_2.
\end{split}
\end{equation}

Thus, by the Calder\'on commutator estimate 

\begin{equation}\label{mathB1}
\|\mathcal B_1\|\les \|\p_x \lanxN\|_{L^\infty}\|x^3 w\|\les P.
\end{equation}

Also, from \eqref{ma3} 

\begin{equation}
\begin{split}\label{mathB2}
\|\mathcal B_2\|\les \|\lanxN x^2 w\|+\|\lanxN x^3 \p_x w\|
                \les P+\|\lanx^4 \lanxN w\|.
\end{split}
\end{equation}



hypotheses \eqref{x2phi}--\eqref{I3} and identity \eqref{x2w} imply that

\begin{equation}\label{mathC}
\begin{split}
\mathcal C \les \ \|\h (x^2w)\|+\| x \h (x^2w)\|
  \les \ \| x^2 w\|+\| \h (x^3 w)\|
  \les \ \| x^2 w\|+\| x^3 w\| 
  \les \ P.
\end{split}
\end{equation}

Hence, by \eqref{ma}--\eqref{mathC}

\begin{equation}
\mathcal A \les \|x^4 \lanxN^{1+\theta}w\|+P.
\end{equation}

About the nonlinear term in \eqref{intwn} we see that

\begin{equation}
\begin{split}
\int \lanxN^{2}x^8w(u_1\p_x w+w\p_x u_2)=&-\frac12\int \big(\p_x(x^8 \lanxN^{2})u_1+x^8 \lanxN^{2}\p_x u_1\big)w^2+\\
&+\int x^8 \lanxN^{2}w^2 \p_x u_2\\
\les&  \ (\|u_1\|_{L^\infty}+\|\p_x u_2\|_{L^\infty})\|x^4 \lanxN w\|^2+P^2\\
\les& \ P (P+\|x^4 \lanxN w\|^2),
\end{split}
\end{equation}
where above we used integration by parts, Sobolev's embedding and the inequality 
 $$|\p_x(x^8 \lanxN^{2})|\les (1+x^8)\lanxN^{2}.$$
 
By the above inequalities we conclude that

\begin{equation*}
\frac{d}{dt}\|\lanxN x^4 w\|^2 \les (1+P)\|\lanxN x^4 w\|^2+P^2.
\end{equation*}

Thus, by Gronwall's Lemma 

\begin{equation*}
\sup_{[-T,T]}\|\lanxN x^4 w\|\leq c(T).
\end{equation*} 

On the other hand, we obtain $w\in C([-T,T];Z_{s,r})$ by similar ideas presented in \cite{AP}.

This finishes the proof.

\end{proof}

The following result allows us to extend the Theorem D, by assuming an additional decay on initial data of the order $3/2^{-}$.   

\begin{proof}[Proof of Theorem \ref{low5theta}]

Suppose that $\theta\in (0,1/2)$. Let $u_1$ and $u_2$ solutions of the IVP \eqref{gbo} with initial data $\phi$ and $\varphi$, respectively. As before, let $\sigma=\phi-\varphi$ and  $w=u_1-u_2$. Thus, Theorems A and \ref{low5} imply that, for all $T>0$, the constant
\begin{equation}
\tilde P= \sup_{[-T,T]}\{\|u_1(t)\|_{H^s}+\|u_2(t)\|_{H^s}+\|w(t)\|_{Z_{s,5}}\}<\infty.
\end{equation}

In a similar way to the proof of the Theorem \ref{low5} we obtain
\begin{equation}\label{intwnthe}
\begin{split}
\frac12 \|\lanxN^{1+\theta}x^4 w\|^2+\int \lanxN^{1+\theta}x^4\h \p_x^2 w\lanxN^{1+\theta}x^4 w+\int \lanxN^{2+2\theta}&x^8(u_1\p_x w+w\p_x u_2)=0,
\end{split}
\end{equation}

where

\begin{equation}
\begin{split}
\lanxN^{1+\theta}x^4 \h \p_x^2 w&=\lanxN^{1+\theta}\h\p_x^2(x^4 w)-8\lanxN^{1+\theta}\h \p_x (x^3 w)+4\lanxN^{1+\theta}\h (x^2w)\\
&:= \ \tilde {\mathcal{A}}+\tilde {\mathcal B}+\tilde{\mathcal C}.
\end{split}
\end{equation}

Next, we will deal with the terms above.

Then

\begin{equation}\label{mathe}
\begin{split}
\tilde {\mathcal{A}}&=[\lanxN^{1+\theta};\h]\p_x^2(x^4 w)+\h(\lanxN^{1+\theta}\p_x^2(x^4 w))\\
&=\  \tilde{\mathcal A_1} +\h\p_x^2 (\lanxN^{1+\theta}x^4 w)-2\h (\p_x \lanxN^{1+\theta}\p_x(x^4 w))-\h(\p_x^2\lanxN^{1+\theta}x^4 w)\\
&= \ \tilde{\mathcal A_1}+\cdots+\tilde{\mathcal A_4}.
\end{split}
\end{equation}

Thus, by the Calder\'on commutator estimate we conclude

\begin{equation}\label{ma1the}
\|\tilde {\mathcal A_1}\|\les \|\p_x^2\lanxN^{1+\theta}\|_{L^\infty}\|x^4 w\|\les \tilde P.
\end{equation}

Also, 

\begin{equation}\label{ma4the}
\|\tilde{\mathcal A_4}\|\les \|\p_x^2\lanxN^{1+\theta}x^4 w\|\les \|\p_x^2\lanxN^{1+\theta}\|_{L^\infty}\|x^4 w\|\les \tilde P.
\end{equation}

By returning the term $\tilde{\mathcal A_2}$ in \eqref{intwn} we obtain

\begin{equation}\label{ma2the}
\int \h \p_x^2 (\lanxN^{1+\theta}x^4w)\lanxN^{1+\theta}x^4 w=0.
\end{equation}
To estimate $\tilde{\mathcal A_3}$, the inequality $|x\p_x \lanxN|\les \lanxN $ yields us 

\begin{equation}\label{ma3the}
\begin{split}
\|\tilde{\mathcal A_3}\| &\les \ \|\lanxN^\theta \p_x \lanxN x^3 w\|+\|\lanxN^\theta \p_x \lanxN x^4 \p_x w\|\\
        &\les \ \|\lanxN^{1+\theta}x^2 w\|+\underbrace{\|\lanxN^{1+\theta}x^3 \p_x w\|}_{\tilde F}\\
        &\les \ \tilde P+\tilde F.
\end{split}
\end{equation}

Using \cite[Theorem 4]{FonsecaPonce} and  identity $\widehat{\p_x (x^3 w)}(0,t)=0$, the term $\tilde{\mathcal B}$ can be estimated as 

\begin{equation}\label{Bthe}
\begin{split}
\tilde{\mathcal B}&\les \ \|\lanxN^\theta \h \p_x(x^3 w)\|+\|\lanxN^\theta x \h \p_x(x^3 w)\|\\
 &\les \ \|\lanxN^{\theta}\p_x(x^3 w)\|+\|\lanxN^\theta \h(x \p_x(x^3 w))\|\\
 &\les \ \|\lanxN^\theta x^2 w\|+\|\lanxN^\theta x^3 \p_x w\|+\|\lanxN^\theta x^3 w\|+\|\lanxN^\theta x^4 \p_x w\|\\
 &\les \ \tilde P+\underbrace{\|\lanxN^\theta \lanx^{4} \p_x w\|}_{\tilde F}.
\end{split}
\end{equation}

By Lemma \ref{interx} (with $\alpha=5(1+\theta)$ and $b=5$)
\begin{equation}\label{Fthe}
\begin{split}
\tilde F&\les \ \|J(\lanxN^\theta \lanx^4 w)\|+\|x^4 w\|+\|\lanxN^\theta \lanx^3 w\|\\
 &\les \ \|J^{1+\theta}(\lanx^4 w)\|+\|\lanxN^{1+\theta}\lanx^4 w\|\\
 &\les \ \|J^{5(1+\theta)} w\|+\|\lanx^5 w\|+\|\lanxN^{1+\theta}\lanx^4 w\|\\
 &\les \ \tilde P+\|\lanxN^{1+\theta} x^4 w\|.
\end{split}
\end{equation}

Also, the Theorem 4 in \cite{FonsecaPonce}, equality \eqref{x2w} and \eqref{x2phi}--\eqref{I3} yields us

\begin{equation}\label{Cthe}
\begin{split}
\tilde {\mathcal C}& \les \ \|\lanxN^\theta\h (x^2w)\|+\|\lanxN^\theta x \h (x^2w)\|\\
  &\les \ \|\lanxN^\theta x^2 w\|+\|\lanxN^\theta \h (x^3 w)\|\\
  &\les \ \|\lanxN^\theta x^2 w\|+\|\lanxN^\theta x^3 w\| \\
  &\les \ \|\lanxN^{1+\theta}\lanx^4 w\|\\
  &\les  \ \tilde P+\|\lanxN^{1+\theta} x^4 w\|.
\end{split}
\end{equation}

Hence, by \eqref{mathe}--\eqref{Cthe}

\begin{equation}
\tilde{\mathcal A} \les \|x^4 \lanxN^{1+\theta}w\|+\tilde P.
\end{equation}

About the nonlinear term in \eqref{intwnthe} we see that

\begin{equation}
\begin{split}
\int \lanxN^{2+2\theta}x^8w(u_1\p_x w+w\p_x u_2)=&-\frac12\int \big(\p_x(x^8 \lanxN^{2+2\theta})u_1+x^8 \lanxN^{2+2\theta}\p_x u_1\big)w^2+\\
&+\int x^8 \lanxN^{2+2\theta}w^2 \p_x u_2\\
\les&  \ (\|u_1\|_{L^\infty}+\|\p_x u_2\|_{L^\infty})\|x^4 \lanxN^{1+\theta}w\|^2+\tilde {P}^2\\
\les& \ \tilde P (\tilde P+\|x^4 \lanxN^{1+\theta}w\|^2),
\end{split}
\end{equation}
where above we used integration by parts, Sobolev's embedding and the inequality 
 $$|\p_x(x^8 \lanxN^{2+2\theta})|\les (1+x^8)\lanxN^{2+2\theta}.$$
 
By the above inequalities we conclude that

\begin{equation*}
\frac{d}{dt}\|\lanxN^{1+\theta}x^4 w\|^2 \les (1+\tilde P)\|\lanxN^{1+\theta}x^4 w\|^2+\tilde {P}^2.
\end{equation*}

Thus, Gronwall's Lemma yields us

\begin{equation*}
\sup_{[-T,T]}\|\lanxN^{1+\theta}x^4 w\|\leq \tilde c(T).
\end{equation*} 

\end{proof}



\end{document}